\documentclass[1p,preprint]{elsarticle}

\usepackage{relsize,amsfonts}
\usepackage{lineno}
\modulolinenumbers[5]
\usepackage{tikz}
\usepackage{framed}
\usepackage{amsmath, amssymb, amsthm}
\usepackage{graphicx}
\usepackage{relsize,amsfonts}
\newcommand\bigexists{%
  \mathop{\lower0.75ex\hbox{\ensuremath{%
    \mathlarger{\mathlarger{\mathlarger{\mathlarger{\exists}}}}}}}%
  \limits}
  \usepackage{hyperref}
\usepackage{xcolor}
\usepackage{mathrsfs}
\usepackage[nameinlink]{cleveref}
\usepackage{quiver}
\usepackage{tikz-cd}
\usepackage{todonotes}
\newcommand{\dav}[1]{\todo[color=blue!30,inline,caption={}]{\textbf{D:} #1}}
\newcommand{\sam}[1]{\todo[color=red!30,inline,caption={}]{\textbf{S:} #1}}

\newcommand{\ovln}[1]{\overline{#1}}

\newcommand{\angbr}[2]{\langle #1,#2 \rangle}

\newcommand{\freccia}[3]{#2\colon#1 \to #3}

\newcommand{\frecciainj}[3]{\xymatrix{#2 \colon #1  \ar@{^{(}->}[r] &  #3}}

\newcommand{\frecciasopra}[3]{#1\xrightarrow{#2} #3}

\newcommand{\pbmorph}[2]{#1^{\ast}#2} 

\newcommand{\duefreccia}[3]{\xymatrix@C=0.5cm{#2 \colon #1  \ar@{=>}[r] &  #3}}

\newcommand{\duemorfismo}[6]{\xymatrix@+1pc{
#1^{\op} \ar[rrd]^#2_{}="a" \ar[dd]_{#3^{\op}}\\
&& \infsl\\
#5^{\op}  \ar[rru]_#6^{}="b"
\ar_{#4}  "a";"b"}}
\newcommand{\comsquare}[8]{ \xymatrix@+1pc{ 
#1 \ar[r]^{#5} \ar[d]_{#6} & #2 \ar[d]^{#7} \\
#3 \ar[r]_{#8} & #4 
}}
\newcommand{\pullback}[8]{ \xymatrix@+1pc{ 
#1 \pullbackcorner \ar[r]^{#5} \ar[d]_{#6} & #2 \ar[d]^{#7} \\
#3 \ar[r]_{#8} & #4 
}}
\newcommand{\quadratocomm}[8]{ \xymatrix@+1pc{ 
#1 \ar[r]^{#5} \ar[d]_{#6} & #2 \ar[d]^{#7} \\
#3 \ar[r]_{#8} & #4 
}}
\newcommand{\comsquarelargo}[8]{ \xymatrix@+1pc{ 
#1 \ar[rr]^{#5} \ar[d]_{#6} && #2 \ar[d]^{#7} \\
#3 \ar[rr]_{#8} && #4 
}}
\newcommand{\parallelmorphisms}[4]{\xymatrix@+1pc{
#1 \ar @<+4pt>[r]^{#2} \ar @<-4pt>[r]_{#3} & #4
}}
\newcommand{\relation}[4]{\xymatrix@+1pc{
\angbr{#2}{#3}\colon #1 \ar @<+4pt>[r] \ar @<-4pt>[r] & #4
}}
\newcommand{\frecceparalleleopposte}[4]{\xymatrix@+1pc{
#1 \ar@<+4pt>[r]^{#2} \ar@<-4pt>@{<-}[r]_{#3} & #4
}}
\newcommand{\equalizer}[6]{\xymatrix@+1pc{
#1 \ar[r]^{#2} & #3 \ar @<+4pt>[r]^{#4} \ar @<-4pt>[r]_{#5} & #6
}}
\newcommand{\coequalizer}[6]{\xymatrix@+1pc{
 #1 \ar @<+4pt>[r]^{#2} \ar @<-4pt>[r]_{#3} & #4 \ar[r]^{#5} & #6
}}

\newcommand{\sottoggetto}[2]{\xymatrix{
#1 \ar@{>->}[r] & #2
}}

\newcommand{\pullbackcorner}[1][ul]{\save*!/#1+1.2pc/#1:(1,-1)@^{|-}\restore}

\def\pr{\pi}
\def\id{\operatorname{ id}}
\def\op{\operatorname{ op}}

\def\mC{\mathcal{C}}

\def\hey{\mathsf{Hey}}
\def\infsl{\mathsf{InfSl}}
\newcommand{\pca}[1]{\mathbb{#1}}

\def\set{\mathsf{Set}}

\newcommand{\doctrine}[2]{#2\colon #1^{\op}\longrightarrow\infsl}
\newcommand{\hyperdoctrine}[2]{#2\colon #1^{\op}\longrightarrow\hey}

\newcommand{\reglex}[1]{(#1)_{\mathsf{reg}/\mathsf{lex}}}
\newcommand{\loc}{\mathsf{A}}

\newcommand{\supercomp}{\mathsf{S}}
\newcommand{\compexfull}[1]{{#1}^{\exists_{\mathsf{f}}}}

\newcommand{\powerset}{\mathscr{P}}

\def\sP{\mathsf{P}}
\usepackage[all]{xy}
\theoremstyle{plain} 
\newtheorem{theorem}{Theorem}[section]
\newtheorem{corollary}[theorem]{Corollary}
\newtheorem{lemma}[theorem]{Lemma}
\newtheorem{proposition}[theorem]{Proposition}
\newtheorem{property}[theorem]{Property}

\theoremstyle{definition} 
\newtheorem{definition}[theorem]{Definition}

\newtheorem{remark}[theorem]{Remark}
\newtheorem{example}[theorem]{Example}
\newcommand{\lar}[1]{\mathlarger{\mathlarger{\mathlarger{#1}}}}

\title{On categorical structures arising from implicative algebras: \\from topology to assemblies}
\author{Samuele Maschio, Davide Trotta}
\date{}
\begin{document}
\begin{abstract}
Implicative algebras have been recently introduced by Miquel in order to provide a unifying notion of model, encompassing the most relevant and used ones, such as realizability (both classical and intuitionistic), and forcing. In this work, we initially approach implicative algebras as a generalization of locales, and we extend several topological-like concepts to the realm of implicative algebras, accompanied by various concrete examples. Then, we shift our focus to viewing implicative algebras as a generalization of partial combinatory algebras. We abstract the notion of a category of assemblies, partition assemblies, and modest sets to arbitrary implicative algebras, and thoroughly investigate their categorical properties and interrelationships.
\end{abstract}

\maketitle

\tableofcontents
\section{Introduction}

The notion of \emph{implicative algebra} has been recently introduced by Miquel \cite{miquel_2020} as a simple algebraic tool to encompass important model-theoretic constructions. These constructions include those underlying forcing and realizability, both in intuitionistic and classical logic. 
 In a subsequent work\cite{miquel_2020_2}, Miquel further reinforces the previously demonstrated outcome by showing that every ``well-behaved $\set$-based semantics'' can be presented as a specific instance of a model within the context of an implicative algebra.

To reach this goal, implicative algebras were initially situated within the categorical setting of $\set$-based \emph{triposes} \cite{TT,TTT}, demonstrating that every implicative algebra induces a $\set$-based tripos~\cite[Thm. 4.4]{miquel_2020}. This result can be seen as a particular case of a more general result presented in \cite[Sec. 5.3]{santos_frey_guillermo_malherbe_miquel_2017} based on the notion of \emph{implicative ordered combinatory algebra.}, since implicative algebras are a particular instance of such a notion. Subsequently, it was proven  that every $\set$-based tripos is isomorphic to an implicative one \cite[Thm. 1.1]{miquel_2020_2}.

It is important to note that  Miquel's results represent the culmination of a series of studies aimed at showing, from an abstract perspective, the essential common features between various realizability-like models and localic models, utilizing categorical tools.

In particular, Hofstra introduced in \cite{Hofstra2006} the notion of \emph{basic combinatory
objects} (BCOs) to encompass (ordered) PCAs and locales, and he provided a characterisation of triposes arising as triposes for ordered PCAs (with filters). The non-ordered version of BCOs, known as \emph{discrete combinatory objects}, was then introduced by Frey in \cite{FreyRT} and employed to provide an ``intrinsic'' or ``extensional'' characterization of realizability toposes.

The main objective of this work is to further explore the abstract perspective that aims to unify realizability-like interpretations and forcing-like (or localic) interpretations. This is achieved by focusing on two aspects: generalizing several topological notions from locale theory to implicative algebras, and employing these new notions to formally define a notion of \emph{category of assemblies} and \emph{category of partitioned assemblies} for implicative algebras. This generalization extends the existing cases of categories of partitioned assemblies and assemblies for a PCA \cite{HYLAND1982165,FreyRT,carboni88}.
In the first part of this work, we concentrate on generalizing standard notions such as \emph{supercompact}, \emph{indecomposable} and \emph{disjoint} elements at the level of implicative algebras. We study these notions in several examples. 
One significant challenge when generalizing localic-like notions to implicative algebras  is the need to consider a suitable form of \emph{uniformity}. This is necessary because, unlike the localic case, the separator of an arbitrary implicative algebra may have multiple elements. This is the case, for example, for the separators of several implicative algebras arising from various forms of realizability. This fact makes, for example, the transition from a point-wise notion to its generalization via indexed sets non-trivial. 

This abstract framework lays the groundwork for the second part of this work, in which we extend various notions derived from realizability to implicative algebras. This extension enables us to offer a topological-like interpretation of these notions.

After completing this initial study, which draws inspiration from the perspective of implicative algebras as generalizations of locales, we then proceed to examine implicative algebras in relation to the broader context of \emph{partial combinatory algebras} (PCAs). This leads us to explore the abstraction of concepts such as assemblies, partitioned assemblies, and modest sets within this framework.

The problem of generalizing these notions from PCAs to arbitrary implicative algebras can be addressed from different perspectives: for example, since assemblies are pairs $(X,\psi)$ where $\freccia{X}{\varphi}{\powerset^* (R)}$ is a function from $X$ to the non-empty powerset of the  PCA $R$, one could try to generalize this notion by defining an assembly for an implicative algebra $\mathbb{A}=(\mathcal{A},\leq,\to,\Sigma)$ as a pair $(X,\psi)$ where  $\freccia{X}{\psi}{\Sigma}$ is a function from the set $X$ to the separator $\Sigma$ (since, for the implicative algebra associated with a PCA, we have that $\Sigma:=\powerset^* (R)$). This approach has been recently used in \cite{CMW}. A second reasonable attempt, based on the fact that every realizability topos can be presented as the $\mathsf{ex/reg}$-completion of the category $\mathbf{Asm}(R)$ of assemblies of its PCA, \cite{RR,SFEC,van_Oosten_realizability}, could be that of defining the category of assemblies for an implicative algebra as the regular completion (in the sense of \cite{UEC}) of the implicative tripos associated with the PCA.
Again, this generalization would allow us to recognize the ordinary category of assemblies as a particular case, since every tripos-to-topos can be presented as the $\mathsf{ex/reg}$-completion of the regular completion of the tripos.

The first solution mainly depends on the ``explicit'' description of an assembly of a PCA, while the second one is based on the abstract properties of such a category.

In this paper, we propose a different approach: instead of focusing on the explicit description of an assembly or on the universal property of the category of assemblies, we aim to identify the \emph{logical properties} that  uniquely identify assemblies in realizability, and then define an arbitrary assembly of an implicative algebra as a pair $(X,\psi)$ where $\psi$ is a predicate of the implication tripos satisfying the logical properties we have identified.

The inspiration for this kind of abstraction is the characterization of predicates determining \emph{partitioned assemblies} presented in \cite{MaiettiTrotta21} in terms of \emph{full existential free elements}, and independently introduced in \cite{Frey2014AFS,Frey2020} via the notion of $\exists$-\emph{primes}: in these works, it has been proved that the functions  $\freccia{X}{\varphi}{R}$ where $R$ is a PCA, i.e. those used to define a partitioned assembly on $X$, correspond exactly to the predicates $\freccia{X}{\phi}{\powerset(R)}$ of the realizability tripos of $R$ satisfying  the following property: 
\begin{property}\label{eq_intro_partition}
Whenever a sequent 
   \[ \phi(x)\vdash  \exists y\in Y\,(f(y)=x\wedge \sigma(y))\]
is satisfied (in the internal language of realizability tripos), there exists a \emph{witness} function $g$ such that $\phi(x)\vdash  \sigma (g(x))$, and this property is preserved by substitutions.
  \end{property}
Following this approach,  we observe that the functions  $\freccia{X}{\varphi}{\powerset^* (R)}$, i.e.\ those used to define an assembly on $X$, correspond exactly to the predicates $\freccia{X}{\phi}{\powerset(R)}$ of the realizability tripos of $R$ satisfying the following property: 
\begin{property}\label{eq_assemb_intro}
    Whenever a sequent  
   \[ \phi (x)\vdash \exists! z\in Z\, \sigma (x,z)\]
where $\sigma$ is a \emph{functional predicate} is satisfied (in the internal language of realizability tripos), there exists a \emph{witness} function $g$ such that $\phi (x)\vdash \sigma (x,g(x))$,  and this property is preserved by substitutions.
\end{property}

The fact that these characterizations do not depend on any explicit description of an assembly or a partitioned assembly, but just on their \emph{logical properties}, makes them easy to generalize to arbitrary triposes.

Therefore,  we define an assembly for an implicative algebra $\mathbb{A}=(\mathcal{A},\leq,\to,\Sigma)$  as a pair $(X,\psi)$  where $X$ is a set and $\freccia{X}{\psi}{\mathcal{A}}$ is a predicate of the implicative tripos associated with $\mathbb{A}$ satisfying the property \eqref{eq_assemb_intro} (in the internal language of the implicative tripos), while we will say that  $\freccia{X}{\psi}{\mathcal{A}}$ is a partitioned assembly if $\freccia{X}{\psi}{\mathcal{A}}$ is a predicate of the implicative tripos associated with $\mathbb{A}$ satisfying the property \eqref{eq_intro_partition}.
Based on our previous analysis, we show that these notions correspond exactly to the generalization of the notion of indecomposable and supercompact elements, respectively, for implicative algebras.

A second crucial insight we propose here regards the concept of \emph{morphism of assemblies} and its generalizations. After introducing a notion of morphism of assemblies following the same idea used in realizability,  where morphisms are defined as 
$\set$-functions, we show that our notion of category of assemblies is equivalent to the subcategory that we called of \emph{strognly trackable objects} of the category of \emph{functional relations} associated with the implicative tripos (i.e. its regular completion in the sense of \cite{UEC}), namely objects such that every functional relation having one of these objects as domain is \emph{tracked} by a unique $\set$-based function. Then, we prove that the category of partitioned assemblies is exactly the subcategory of strongly trackable objects which are \emph{regular projectives} of the category of functional relations associated with the implicative tripos.

Finally, we conclude by studying some basic categorical properties of these categories. In particular, it is worth recalling that in realizability the category of assemblies is regular, and it happens to be equivalent to the regular completion (in the sense of \cite{SFEC}), of its full subcategory of partition assemblies. However, this connection between assemblies and partition assemblies does not hold in general for the case of an arbitrary implicative algebra. More generally, for an arbitrary implicative algebra, the category of assembly is not regular, and the category of partition assemblies has no finite limits. 

Taking inspiration again from \cite{MaiettiTrotta21}, we present necessary and sufficient conditions allowing us to understand when  a category of assemblies for an implicative algebra is regular and it is the regular completion of the category of partition assemblies.

\section{Implicative algebras}

In this section we recall the definition of implicative algebras introduced in \cite{miquel_2020}.
\subsection{Definition}
\begin{definition}[implicative structure]
   An \textbf{implicative structure} $\mathbb{A}=(\mathcal{A},\leq, \to)$ is a complete lattice $(\mathcal{A},\leq)$ equipped with a binary operation $(a,b)\mapsto (a\to b)$ called \textbf{implication} of $\mathcal{A}$ satisfying the following two axioms:
   \begin{itemize}
      \item if $a'\leq a$ and $b\leq b'$ then $a\to b\leq a'\to b'$;
      \item $a\to \bigwedge_{b\in B}b=\bigwedge_{b\in B}(a\to b)$, for every $a\in \mathcal{A}$ and every subset $B\subseteq \mathcal{A}$.
   \end{itemize}
\end{definition}

\begin{definition}[separator]
   Let $\mathbb{A}=(\mathcal{A},\leq, \to)$ be an implicative structure. A \textbf{separator} is a subset $\Sigma\subseteq \mathcal{A}$ satisfying the following conditions for every $a,b\in \mathcal{A}$:
   \begin{itemize}
      \item if $a\in \Sigma$ and $a\leq b$ then $b\in \Sigma$;
      \item $\mathbf{k}^{\mathbb{A}}:=\bigwedge_{a,b\in \mathcal{A}}(a \to b \to a)$ is an element of $\Sigma$;
      \item $\mathbf{s}^{\mathbb{A}}:=\bigwedge_{a,b,c\in \mathcal{A}}((a \to b \to c)\to (a\to b)\to a\to c)$ is an element of $\Sigma$;
      \item if $(a\to b)\in \Sigma$ and $a\in \Sigma$ then $b\in \Sigma$.
   \end{itemize}
\end{definition}
The intuition is that a separator $\Sigma\subseteq \mathcal{A}$ determines a particular ``criterion of truth'' within the implicative structure $(\mathcal{A},\leq, \to)$, generalizing the notion of filters for Heyting algebras. 


\begin{definition}[implicative algebra]
   We call an \textbf{implicative algebra} an implicative structure $(\mathcal{A},\leq,\to)$ equipped with a separator $\Sigma\subseteq \mathcal{A}$. In such a case the implicative algebra will be denoted as $(\mathcal{A},\leq,\rightarrow,\Sigma)$.
\end{definition}

\subsection{Some examples of implicative algebras}\label{subse_examples_IA}
\subsubsection*{Complete Heyting algebras}
If $\mathbb{H}=(H,\leq)$ is a complete Heyting algebra with Heyting implication $\rightarrow$, we can see it as an implicative algebra $(H,\leq,\rightarrow,\{\top\})$ where $\top$ is the maximum of $\mathbb{H}$. 
\subsubsection*{Realizability}
If $\mathcal{R}=(R,\cdot)$ is a (total) combinatory algebra (CA) (see e.g. \cite{van_Oosten_realizability}), then we can define an implicative algebra 
from it by considering the $4$-tuple, $(\mathcal{P}(R),\subseteq, \Rightarrow,\mathcal{P}(R)\setminus\{\emptyset\})$ where $A\Rightarrow B:=\{r\in R|\,r\cdot a\in B\textrm{ for all }a\in A\}$ for every $A,B\subseteq R$.

\subsubsection*{Nested realizability}
Nested realizability tripos is considered in \cite{birkedalvanoosten02,maschiostreicher15} in order to study some aspects of modified realizability and relative realizability (see \cite{vanoosten97,van_Oosten_realizability}).
We consider here only the total case in order not to make the notation heavy. The same we will do in the next examples.
Let $\mathcal{R}=(R,\cdot)$ be a combinatory algebra and let $\mathcal{R}_{\#}=(R_\#,\cdot_{\#})$ be one of its sub-combinatory algebras, that is $R_\# \subseteq R$, $a\cdot_{\#} b=a\cdot b$ for every $a,b\in R_\#$ and  $\mathbf{k}$, $\mathbf{s}$ in $\mathcal{R}$ can be chosen to be elements of $R_{\#}$. We can define an implicative algebra $\mathbb{A}^n_{\mathcal{R},\mathcal{R}_{\#}}:=(P_{\mathcal{R},\mathcal{R}_{\#}},\subseteq_n, \Rightarrow_n,\Sigma_{\mathcal{R},\mathcal{R}_{\#}})$ as follows:
\begin{enumerate}
\item $P_{\mathcal{R},\mathcal{R}_{\#}}:=\{(X_a,X_p)\in \mathcal{P}(R_\#)\times \mathcal{P}(R)|\,X_a\subseteq X_p\}$
\item $(X_a,X_p)\subseteq_n (Y_a,Y_p)$ if and only if $X_a\subseteq Y_a$ and $X_p\subseteq Y_p$;
\item $(X_a,X_p)\Rightarrow_n (Y_a,Y_p):=((X_a\Rightarrow_{\#}Y_a)\cap (X_p\Rightarrow Y_p), X_p\Rightarrow Y_p )$
\item $(X_a,X_p)\in \Sigma_{\mathcal{R},\mathcal{R}_{\#}}$ if and only if $X_a\neq \emptyset$.
\end{enumerate}
\subsubsection*{Modified realizability}
Let $\mathcal{R}=(R,\cdot)$ be a combinatory algebra and let $\mathcal{R}_{\#}=(R_\#,\cdot_{\#})$ be one of its sub-combinatory algebras and assume there exists $\star\in R_\#$ such that $\star\cdot x=\star$ for every $x\in R$ and $\mathbf{p}\cdot\star\cdot\star=\star$ for $\mathbf{p}$ the pairing combinator defined from fixed $\mathbf{k},\mathbf{s}\in R_\#$.
We can define an implicative algebra as
$$\mathbb{A}^m_{\mathcal{R},\mathcal{R}_{\#},\star}:=(P^m_{\mathcal{R},\mathcal{R}_{\#},\star},\subseteq_n, \Rightarrow_n,\Sigma_{\mathcal{R},\mathcal{R}_{\#}}\cap P^m_{\mathcal{R},\mathcal{R}_{\#},\star} )$$  where
$$P^m_{\mathcal{R},\mathcal{R}_{\#},\star}:=\{(X_a,X_p)\in P_{\mathcal{R},\mathcal{R}_{\#}}|\, \star\in X_p \}$$
\subsubsection*{Relative realizability}
Let $\mathcal{R}=(R,\cdot)$ be a combinatory algebra and let $\mathcal{R}_{\#}=(R_\#,\cdot_{\#})$ be one of its sub-combinatory algebras. We define the relative realizability implicative algebra as follows:
$$\mathbb{A}^r_{\mathcal{R},\mathcal{R}_{\#}}:=(\mathcal{P}(R),\subseteq, \Rightarrow,\Sigma^r_{\mathcal{R},\mathcal{R}_{\#}})$$ 
where
$$\Sigma^r_{\mathcal{R},\mathcal{R}_{\#}}:=\{X\in \mathcal{P}(R)|\,X\cap R_\#\neq \emptyset\}$$ 
\subsubsection*{Partial cases} Notice that one can also consider the previous cases in which the binary operations of the combinatory algebras involved are \emph{partial}. In this case we do not obtain implicative algebras, but \emph{quasi-implicative algebras}. However by considering a notion of completion which can be found in \cite{miquel_2020} one can obtain implicative algebras from them.  The choice of presenting just \emph{total} PCAs instead of the more traditional and general notion is motivated by the crucial result of Miquel, i.e.\ the fact that every quasi-implicative tripos associated with a (partial) PCA is isomorphic to an implicative one. We refer to \cite[Sec. 4]{miquel_2020} for all details.

\subsubsection*{Classical realizability}Let $\mathcal{K}=(\Lambda,\Pi,@,\cdot,\mathbf{k}_{-},\mathbf{K},\mathbf{S},\mathbf{cc},\mathsf{PL},\perp)$ be an abstract Krivine structure (see \cite{streicher13,miquel_2020}). One can define an implicative algebra as follows
$(\mathcal{P}(\Pi),\supseteq, \rightarrow, \Sigma)$
where $X\rightarrow Y:=\{t\cdot\pi|\,t\in X^\perp,\pi\in Y\}$, where $X^{\perp}:=\{t\in \Lambda|\,t\perp \pi \textrm{ for every }\pi\in Y\}$ and $\Sigma=\{X\in \mathcal{P}(\Pi)|\,X^\perp\cap PL\neq \emptyset\}$.
\subsection{The encoding of $\lambda$-terms in an implicative algebra}
In any implicative algebra $\mathbb{A}=(\mathcal{A},\leq,\to,\Sigma)$ one can define  a binary application as follows for every $a,b$ in $\mathcal{A}$
$$a\cdot b:=\bigwedge \{x\in \mathcal{A}|\,a\leq b\rightarrow x\}.$$
Using this one can encode closed $\lambda$-terms with parameters in $\mathcal{A}$ as follows:
\begin{enumerate}
\item $a^{\mathbb{A}}:=a$ for every $a\in \mathcal{A}$; 
\item $(ts)^{\mathbb{A}}:=t^{\mathbb{A}}\cdot s^{\mathbb{A}}$;
\item $(\lambda x.t)^{\mathbb{A}}:=\bigwedge_{a\in \mathcal{A}}(a\rightarrow (t[a/x])^{\mathbb{A}}$
\end{enumerate}

 A nice result is that if $t$ $\beta$-reduces to $s$, then $t^{\mathcal{A}}\leq s^{\mathcal{A}}$.
Moreover, if $t$ is a pure $\lambda$-term with free variables $x_1,...,x_n$ and $a_1,...,a_n\in \Sigma$, then $(t[a_1/x_1,...,a_n/x_n])^{\mathbb{A}}\in \Sigma$.

Finally, we can notice that $\mathbf{k}^{\mathbb{A}}$ and $\mathbf{s}^{\mathbb{A}}$ are exactly the interpretations of the $\lambda$-terms $\mathbf{k}:=\lambda x.\lambda y.x$ and $\mathbf{s}:=\lambda x.\lambda y.\lambda z.xz(yz)$ as shown in \cite[Prop. 2.24]{miquel_2020}.

\subsection{The calculus of an implicative algebra}
In every implicative algebra $\mathbb{A}=(\mathcal{A},\leq,\to,\Sigma)$ we can define first-order logical operators in such a way that we obtain a very useful calculus. 

In particular if $a,b$ are in $\mathbb{A}$ and $(c_i)_{i\in I}$ is a family of elements of $\mathbb{A}$ we can define
 $$a\times b:=\bigwedge_{x\in \mathcal{A}}\left((a\rightarrow (b\rightarrow x))\rightarrow x\right)\qquad a+b:=\bigwedge_{x\in \mathcal{A}}\left((a\rightarrow x)\rightarrow ((b\rightarrow x)\rightarrow x)\right)$$
 $$\bigexists_{i\in I}c_i:=\bigwedge_{x\in \mathcal{A}}\left(\bigwedge_{i\in I}(c_i\rightarrow x)\rightarrow x\right)$$
As shown in \cite{miquel_2020}, the following rules hold:
{\small $$
\cfrac{(x:A)\in \Gamma}{\Gamma\vdash x:A}\qquad \cfrac{}{\Gamma\vdash a:a}\qquad \cfrac{\Gamma\vdash t:a\qquad a\leq b}{\Gamma\vdash t:b}\qquad \cfrac{\Gamma'\leq \Gamma\;\;\Gamma\vdash t:a}{\Gamma'\vdash t:a}\footnote{By $\Gamma'\leq \Gamma$ we mean that for every variable assignment $x:a$ in $\Gamma$ there is $b\leq a$ such that $x:b$ is in $\Gamma'$.}\qquad \cfrac{\mathsf{Free}(t)\subseteq Var(\Gamma)}{\Gamma\vdash t:\top}
$$
$$\cfrac{\Gamma\vdash t:\bot}{\Gamma\vdash t:a}\qquad \cfrac{\Gamma,x:a\vdash t:b}{\Gamma\vdash \lambda x.t:a\rightarrow b}\qquad \cfrac{\Gamma\vdash t:a\rightarrow b\;\;\Gamma\vdash s:a}{\Gamma\vdash ts:b}\qquad \cfrac{\Gamma\vdash t:a\qquad \Gamma\vdash u:b}{\Gamma\vdash \lambda z.ztu:a\times b}$$
$$\cfrac{\Gamma\vdash t:a\times b}{\Gamma\vdash t(\lambda x.\lambda y. x):a}\qquad \cfrac{\Gamma\vdash t:a\times b}{\Gamma\vdash t(\lambda x.\lambda y. y):b}\qquad \cfrac{\Gamma\vdash t:a}{\Gamma\vdash \lambda z.\lambda w.zt:a+b}\qquad \cfrac{\Gamma\vdash u:b}{\Gamma\vdash \lambda z.\lambda w.wt:a+b}$$
$$\cfrac{\Gamma\vdash t:a+b\;\;\Gamma,x:a\vdash u:c\;\;\Gamma,y:b\vdash v:c}{\Gamma\vdash t(\lambda x.u)(\lambda y.v):c}\qquad \cfrac{\Gamma\vdash t:a_{i}\;(\textrm{for all }i\in I)}{\Gamma\vdash t:\bigwedge_{i\in I}a_i}\qquad \cfrac{\Gamma\vdash t:\bigwedge_{i\in I}a_i\;\;\overline{i}\in I}{\Gamma\vdash t:a_{\overline{i}}} $$
$$\cfrac{\Gamma\vdash t:a_{\overline{i}}\;\;\;\overline{i}\in I}{\Gamma\vdash \lambda z.zt:\bigexists_{i\in I}a_i}\qquad \cfrac{\Gamma\vdash t:\bigexists_{i\in I}a_i\;\;\Gamma,x:a_i\vdash u:c\;(\textrm{ for all }i\in I)}{\Gamma\vdash t(\lambda x.u):c}$$
}
where every sequent $\Gamma\vdash t:a$ contains a list of variable declarations $\Gamma:=x_1:a_1,...,x_n:a_n$ where $x_1,...,x_n$ are distinct variables and $a_1,...,a_n\in \mathcal{A}$, a lambda term $t$ containing as free variables at most those in $\Gamma$ and an element $a\in \mathcal{A}$. 

The meaning of such a sequent is that $(t[\Gamma])^{\mathbb{A}}\leq a$ where $t[\Gamma]$ denotes the term obtained from $t$ by performing the substitution indicated by $\Gamma$.

The calculus above is very useful, since using the remarks from the previous subsection, if we deduce that $x_1:a_1,...,x_n:a_n\vdash t:b$, $t$ is a pure $\lambda$-term and $a_1,...,a_n\in \Sigma$, then $b$ is in $\Sigma$ too.

We now state two propositions which can be easily proved by using the calculus above.
\begin{proposition}\label{existentialgen} If $\mathcal{F}$ is the class of set-indexed families of elements of an implicative algebra $\mathbb{A}$, then
$$\bigwedge_{(b_i)_{i\in I}\in \mathcal{F}}\left(\bigexists_{i\in I}b_i\rightarrow \bigvee_{i\in I}b_i\right)\in \Sigma.$$
\end{proposition}
Before proving the next proposition we need to introduce two notions of equality evaluated in an implicative algebra. The first, which is the right one, is equivalent to that presented in \cite{miquel_2020} under the name $\mathbf{id}$. It is defined as follows for every set $J$ and every $j,j'\in J$
$$\delta_{J}(j,j'):=\bigexists_{i=j=j'}\top=\begin{cases}
\top\rightarrow \bot \textrm{ if }j\neq j'\\
\bigwedge_{x\in \mathcal{A}}((\top\rightarrow x)\rightarrow x)\textrm{ if }j=j'.\\
\end{cases}$$
The second one, despite seeming more natural, is not even an equivalence relation in general with respect to the logical calculus of $\mathbb{A}$
$$d_{J}(j,j'):=\begin{cases}\bot\textrm{ if }j\neq j'\\ \top\textrm{ if }j=j'.\\ \end{cases}$$

\begin{proposition}\label{equalitygen}
For every implicative algebra $\mathbb{A}$  we have that
$$\bigwedge_{J\in \set}\bigwedge_{j,j'\in J}\left(\delta_J(j,j')\rightarrow d_{J}(j,j')\right)\in \Sigma.$$
\end{proposition}
We conclude this section by introducing the notation $a\vdash_{\Sigma}b$ which is used as a shorthand for $a\rightarrow b\in \Sigma$ ($a,b\in \mathcal{A}$). The pair $(\mathcal{A},\vdash_{\Sigma})$ is a preorder whose posetal reflection is a Heyting algebra. We write $a\equiv_{\Sigma}b$ when $a\vdash_{\Sigma} b$ and $b\vdash_{\Sigma}a$.

\section{Topological notions in implicative algebras}\label{sec_topological_notions_IA}
The main purpose of this section is to generalize various topological notions at the level of implicative algebras.
We will see that one of the fundamental problems that arise when trying to generalize topological notions at the level of implicative algebra is that of ``uniformity``, i.e. the stability of a given property under reindexing.

In this section we fix an arbitrary implicative algebra $\mathbb{A}=(\mathcal{A},\leq,\rightarrow, \Sigma)$. 
\subsection{Disjoint families}
The first basic, yet fundamental, notion that we want to introduce and establish in the language of implicative algebras is that of \emph{disjointness}. From an algebraic perspective, the property of \emph{being disjoint} for two elements $a,b$ of a complete Heyting algebra $(H,\leq)$ can be simply stated as $a\wedge b=\bot$. This is equivalent to the statement that $a\wedge b\rightarrow \bot=\top$. Since $a\wedge a\rightarrow \top=\top$ for every $a$, one can say that a family $(a_i)_{i\in I}$ of elements of $H$ is pairwise disjoint if 
$\bigwedge_{i,j\in I}(a_i\wedge a_j\rightarrow \delta(i,j))=\top$ where $\delta(i,j)=\top$ if $i=j$ and $\delta(i,j)=\bot$ if $i\neq j$.


Since any implicative algebra has two different notions of ``infimum'', one with respect to $\leq$ (subtyping), the other with respect to $\vdash_\Sigma$ (logical entailment), it seems natural to consider two different notions to abstract such a notion of disjoint family:
\begin{definition}\label{def_disjoint_family}
A family $(a_i)_{i\in I}$ of elements of $\mathbb{A}$ is 
\begin{enumerate}
\item $\wedge$-{\bf disjoint} if
$\bigwedge_{i,i'\in I}(a_i\wedge a_{i'}\rightarrow \delta_I(i,i'))\in \Sigma$
\item $\times$-{\bf disjoint} if
$\bigwedge_{i,i'\in I}(a_i\times a_{i'}\rightarrow \delta_I(i,i'))\in \Sigma$
\end{enumerate}
\end{definition}
\begin{proposition} If a family is $\times$-{\bf disjoint}, then it is $\wedge$-{\bf disjoint}.
\end{proposition}
\begin{proof}
This follows from the fact that $\bigwedge_{a,b\in \mathcal{A}}(a\wedge b\rightarrow a\times b)\in \Sigma$.
\end{proof}
\begin{example}
In any complete Heyting algebra the two notions presented in \Cref{def_disjoint_family} clearly coincide, because $a\wedge b=a\times b$, and they coincide with the notion of pairwise disjoint family of elements of a Heyting algebra above.
\end{example}
\begin{example}
In the case of the implicative algebra associated with a combinatory algebra $(R,\cdot)$, then one can easily check that $\wedge$-disjoint families are families $(A_i)_{i\in I}$ such that $A_i\cap A_j=\emptyset$ for every $i,j\in I$ with $i\neq j$, while $\times$-disjoint families are families $(A_{i})_{i\in I}$ such that at most one of the $A_i$'s is non-empty. 
The same holds for relative realizability implicative algebras.
\end{example}
\begin{example}
    
In the nested realizability implicative algebras a family $(A_i,B_i)_{i\in I}$ is $\wedge$-disjoint if and only if $B_i\cap B_j= \emptyset$ for every $i,j\in I$ with $i\neq j$, while it is $\times$-disjoint if and only if at most one of the $B_i$'s is  non-empty.
\end{example}
\begin{example}
In modified realizability implicative algebras, $\times$-disjoint families are families $(A_i,B_i)_{i\in I}$ in which at most one of the $A_i$'s is non-empty, while  $\wedge$-disjoint families are families $(A_i,B_i)_{i\in I}$ in which $A_i\cap A_j=\emptyset$ for all $i\neq j$ in $I$.

\end{example}

We also introduce the following notion of $\times $-functional family which will be useful later.
\begin{definition}
A two-indexed family $(b_{j}^{i})_{i\in I,j\in J}$ of elements of $\mathbb{A}$ is {\bf $\times$-functional} if 
$$\bigwedge_{i\in I}\bigwedge_{j,j'\in J}(b_j^i\times b^i_{j'}\rightarrow \delta_{J}(j,j'))\in \Sigma$$
\end{definition}
\begin{remark}\label{functional_implies_disjoint_for_fix_i}
    Notice that if $(b_{j}^{i})_{i\in I,j\in J}$ is $\times$-functional, then for every $i\in I$ the family $(b_{j}^{i})_{j\in J}$ is $\times$-disjoint.
\end{remark}
\subsection{Supercompactness}\label{SKsec}
The second notion we aim to abstract in the setting of implicative algebras is that of a \emph{supercompact element} \cite{BANASCHEWSKI199145,pp}. Recall that in the case of a complete Heyting algebra, an element $a$ is said to be \emph{supercompact} if 
$$a\leq \bigvee_{i\in I}b_i$$
implies the existence of an  $\overline{i}\in I$ such that $a\leq b_{\overline{i}}$, for every set-indexed family $(b_i)_{i\in I}$ of elements.

Taking inspiration from this notion, we introduce the following generalization:
\begin{definition}[supercompact element]\label{def_supercompact_element_IA}
An element $a\in \mathcal{A}$ is {\bf supercompact} in $\mathbb{A}$ if for every set-indexed family $(b_i)_{i\in I}$ of elements of $\mathcal{A}$ with $$a\rightarrow \bigexists_{i\in I} b_i\in \Sigma$$ there exists $\overline{i}\in I$ such that $a\rightarrow b_{\overline{i}}\in \Sigma$.
\end{definition}
\begin{remark}\label{supmin} Notice that the minimum $\bot$ can never be supercompact in $\mathbb{A}$. Indeed, if we consider an empty family we always have $\mathbf{k}^{\mathbb{A}}\leq \bot\rightarrow (\top\rightarrow \bot)=\bot\rightarrow \bigexists \emptyset$ from which it follows that $\bot\rightarrow \bigexists \emptyset\in \Sigma$.
\end{remark}
\begin{remark}\label{supmax} Notice that the maximum $\top$ is supercompact if and only if  from $\bigexists_{i\in I}b_i\in \Sigma$, one can deduce the existence of an index $\overline{i}\in I$ such that $b_{\overline{i}}\in \Sigma$. We can consider such a property as a sort of \emph{existence property}. 
Complete Heyting algebras satisfying this property are called supercompact (locales) (see \cite{pp}), while the only complete Boolean algebras safistying this property are the trivial ones (those in which every element is either $\bot$ or $\top$), since for every $a$ we have $a\vee \neg a=\top$. The implicative algebra of realizability in the total case satisfies this property thanks to Proposition \ref{existentialgen} and the fact that the union of a family of sets is non-empty if and only if at least one of them is non-empty. 
The same holds for relative realizability, nested realizability and modified realizability. 
\end{remark}
\begin{remark} One can easily notice that if $a\equiv_\Sigma b$ and $a$ is supercompact, then $b$ is supercompact too.
\end{remark}

If we want to generalize the notion of supercompact element to a notion of supercompact family of elements of $\mathbb{A}$ we have three natural ways:

\begin{definition}
A family $(a_i)_{i\in I}$ of elements of $\mathbb{A}$ is
\begin{enumerate}
\item {\bf componentwise supercompact (cSK)} if $a_i$ is supercompact for every $i\in I$;
\item \textbf{supercompact (SK)} if  for every family of families $((b^i_j)_{j\in J_i})_{i\in I}$ of elements of $\mathcal{A}$ whenever
$$\bigwedge_{i\in I}(a_i\rightarrow \bigexists_{j\in J_i}b^i_j)\in \Sigma$$
there exists $f\in (\Pi i\in I)J_i$ such that
$$\bigwedge_{i\in I}(a_i\rightarrow b^i_{f(i)})\in \Sigma$$
\item {\bf uniformly supercompact ($\mathbf{U\mbox{-}SK}$)} if $(a_{f(k)})_{k\in K}$ is $\mathbf{SK}$ for every function $f:K\rightarrow I$.
\end{enumerate}
\end{definition}
\begin{remark}\label{cSKSK}
One can observe that a family $(a_{i})_{i\in \{\star\}}$ is $\mathbf{SK}$ if and only if $a_{\star}$ is a supercompact element.
\end{remark}

\begin{proposition}\label{carusk}
The following are equivalent for a family $(a_i)_{i\in I}$ of elements of $\mathbb{A}$:
\begin{enumerate}
\item $(a_i)_{i\in I}$ is $\mathbf{U}$-$\mathbf{SK}$;
\item for every family of families of families $(((b^k_j)_{j\in J_K})_{k\in K_i})_{i\in I}$ of elements of $\mathcal{A}$ with $(K_i)_{i\in I}$ a family of pairwise disjoint sets, such that
$$\bigwedge_{i\in I}(a_i\rightarrow \bigwedge_{k\in K_i}\bigexists_{j\in J_k}b^k_j)\in \Sigma$$
there exists a function $g\in (\Pi k\in \bigcup_{i\in I}K_i)J_k$ such that
$$\bigwedge_{i\in I}(a_i\rightarrow \bigwedge_{k\in K_i}b^{k}_{g(k)})\in \Sigma $$ 
\end{enumerate}
\end{proposition}
\begin{proof}
Let $(a_i)_{i\in I}$ satisfy $2.$ and $f:K\rightarrow I$ be a function. Assume that
$$\bigwedge_{k\in K}(a_{f(k)}\rightarrow \bigexists_{j\in J_k}b_j^k)\in \Sigma$$
This implies that 
$$\bigwedge_{i\in I}(a_{i}\rightarrow \bigwedge_{k\in f^{-1}(i)}\bigexists_{j\in J_k}b_j^k)\in \Sigma.$$
Using the fact that the family $(a_i)_{i\in I}$ satisfies $2$, we get the existence of a function $g\in(\Pi k\in K)J_k$ such that
$$\bigwedge_{k\in K}\left(a_{f(k)}\rightarrow b^k_{g(k)}\right)\in \Sigma.$$
Conversely, assume that $(a_{i})_{i\in I}$ is a family such that $(a_{f(j)})_{j\in J}$ is $\mathbf{SK}$ for every function $f:J\rightarrow I$ and that
$$\bigwedge_{i\in I}\left(a_i\rightarrow \bigwedge_{k\in K_i}\bigexists_{j\in J_k}b^k_j\right)\in \Sigma$$
where $(K_i)_{i\in I}$ is a family of pairwise disjoint sets. Let $f:K:=\bigcup_{i\in I}K_i\rightarrow I$ be the function sending each $k\in K_i$ to $i$. Thus,
$$\bigwedge_{k\in K}\left(a_{f(k)}\rightarrow \bigexists_{j\in J_k}b^k_j\right)\in \Sigma.$$
Since $(a_{f(k)})_{k\in K}$ is $\mathbf{SK}$, then there exists $g\in (\Pi k\in K)J_k$ such that
$$\bigwedge_{k\in K}\left(a_{f(k)}\rightarrow b^k_{g(k)}\right)\in \Sigma$$
that is
$$\bigwedge_{i\in I}\left(a_{i}\rightarrow \bigwedge_{k\in K_i}b^k_{g(k)}\right)\in \Sigma$$
\end{proof}
The following proposition is almost an immediate consequence of the definitions.
\begin{proposition}\label{u->c}
If $(a_i)_{i\in I}$ is $\mathbf{U\mbox{-}SK}$, then it is $\mathbf{cSK}$ and $\mathbf{SK}$.
\end{proposition}
\begin{proof}
Let $(a_i)_{i\in I}$ be $\mathbf{U}\mbox{-}\mathbf{SK}$. 
If we consider the identity function $\mathsf{id}_I$, we immediately obtain that $(a_{i})_{i\in I}$ is $\mathbf{SK}$. If we consider the functions from a singleton $\{\star\}$ to $I$ we obtain that $(a_{i})_{i\in I}$ is $\mathbf{cSK}$ by Remark \ref{cSKSK}.
%
\end{proof}

\subsection{Indecomposability}\label{Indsec}
Now we introduce a weaker variant of the notion of supercompact element, that is that of an \emph{indecompomposable element}. Again, the intuition is that an \emph{indecomposable} element $a$ of a complete Heyting algebra (locale) is an object such that whenever
$$a\leq \bigvee_{i\in I}b_i$$
 then there the existence of an  $\overline{i}\in I$ such that $a\leq b_{\overline{i}}$, for every set-indexed family $(b_i)_{i\in I}$ of pairwise \emph{disjoint} elements.
\begin{definition} An element $a$ of $\mathbb{A}$ is {\bf indecomposable} if for every $\times$-disjoint family $(b_i)_{i\in I}$ of elements of $\mathbb{A}$, whenever $a\rightarrow \bigexists_{i\in I}b_i\in \Sigma$, there exists $\overline{i}\in I$ such that $a\rightarrow b_{\overline{i}}\in \Sigma$.
\end{definition}
From the definition, it easily follows that
\begin{proposition}\label{prop_supercomp_implies_indecomposable}
Every supercompact element of $\mathbb{A}$ is also indecomposable.
\end{proposition}
\begin{remark}\label{nminfSK}
Similarly to what happens with supercompactness, if $a\equiv_\Sigma b$ and $a$ is indecomposable, then $b$ is indecomposable. Moreover, $\bot$ can never be indecomposable. 
\end{remark}
\begin{remark}
Notice that if $(b_i)_{i\in I}$ is a $\times$-disjoint family, $a$ is indecomposable, $a\rightarrow b_{i_1}\in \Sigma$ and $a\rightarrow b_{i_2}\in \Sigma$, then $a\rightarrow b_{i_1}\times b_{i_2}\in \Sigma$ from which it follows that $a\rightarrow \delta_{I}(i_1,i_2)\in \Sigma$. If $i_1\neq i_2$, then this means that $a\rightarrow (\top\rightarrow \bot)\in \Sigma$, from which it follows that $a\equiv_{\Sigma} \bot$. But we know that this cannot happen. So $i_1=i_2$. This means that the element $\overline{i}$ in the definition of indecomposable element is in fact unique. 
\end{remark}

In order to generalize the notion of indecomposability to families we consider the following five notions:
\begin{definition} 
Let $I$ be a set. A family $(a_i)_{i\in I}$ of elements of $\mathbb{A}$ is: 
\begin{enumerate}
\item {\bf componentwise indecomposable} (\textbf{cInd}) if $a_i$ is indecomposable for every $i\in I$;
\item {\bf functionally supercompact (fSK)} if for every $\times$-functional family $(b^i_j)_{i\in I, j\in J}$ of elements of $\mathcal{A}$ such that
$$\bigwedge_{i\in I}(a_i\rightarrow \bigexists_{j\in J}b^i_j)\in\Sigma$$
there exists a unique $f:I\rightarrow J$ such that
$\bigwedge_{i\in I}(a_i\rightarrow b^i_{f(i)})\in \Sigma$; 

\item {\bf uniformly functionally supercompact ($\mathbf{U\mbox{-}fSK}$)} if $(a_{f(k)})_{k\in K}$ is functionally supercompact for every function $f:K\rightarrow I$;


\item {\bf weakly functionally supercompact (wfSK)}  if for every $\times$-functional family  $(b^i_j)_{i\in I,j\in J}$ of elements of $\mathcal{A}$ such that
$$\bigwedge_{i\in I}(a_i\rightarrow \bigexists_{j\in J}b^i_j)\in \Sigma$$
there exists $f:I\rightarrow J$ such that
$\bigwedge_{i\in I}(a_i\rightarrow b^i_{f(i)})\in \Sigma$;

\item {\bf uniformly weakly functionally supercompact ($\mathbf{U\mbox{-}wfSK}$)} if $(a_{f(k)})_{k\in K}$ is weakly functionally supercompact for every function $f:K\rightarrow I$.


\end{enumerate}
\end{definition}
 By definition and using arguments similar to that in the proof of Proposition \ref{u->c} we have that:
\begin{proposition}\label{fsk->wfsk}\label{u->c}
For a family  $(a_{i})_{i\in I}$ of elements of $\mathbb{A}$ we have that:
\begin{enumerate}
\item $\mathbf{fSK}\Rightarrow\mathbf{wfSK}$;
\item $\mathbf{U\mbox{-}fSK}\Rightarrow\mathbf{U\mbox{-}wfSK}$
\item $\mathbf{U\mbox{-}fSK}\Rightarrow\mathbf{fSK}$;
\item $\mathbf{U\mbox{-}wfSK}\Rightarrow\mathbf{wfSK}$;
\item $\mathbf{U\mbox{-}wfSK}\Rightarrow\mathbf{cInd}$.
\end{enumerate}
\end{proposition}
Moreover with a proof analogous to that of Proposition \ref{carusk} one can easily prove that:
\begin{proposition}\label{reind}
Let $(a_{i})_{i\in I}$ be a family of elements of $\mathbb{A}$. Then the following are equivalent:
\begin{enumerate}
\item the family $(a_{i})_{i\in I}$ is $\mathbf{U\mbox{-}wfSK}$ (\,$\mathbf{U\mbox{-}fSK}$ respetively);
\item for every family of families of families $(((b^k_j)_{j\in J})_{k\in K_i})_{i\in I}$ of elements of $\mathcal{A}$ with $(K_i)_{i\in I}$ a family of pairwise disjoint sets and $(b^k_j)_{k\in \bigcup_{i\in I}K_i, j\in J}$ $\times$-functional, such that
$$\bigwedge_{i\in I}(a_i\rightarrow \bigwedge_{k\in K_i}\bigexists_{j\in J}b^k_j)\in \Sigma$$
there exists a (respectively unique) function $g:(\Pi k\in \bigcup_{i\in I}K_i)\rightarrow J$ such that
$$\bigwedge_{i\in I}(a_i\rightarrow \bigwedge_{k\in K_i}b^{k}_{g(k)})\in \Sigma $$
\end{enumerate}
\end{proposition}


\begin{proposition}\label{sk->fsk} If $(a_{i})_{i\in I}$ is $\mathbf{SK}$, then it is $\mathbf{fSK}$.
\end{proposition}
\begin{proof}
By definition if a family is $\mathbf{SK}$ then it is $\mathbf{wfSK}$. In order to conclude we need to show that if $(a_i)_{i\in I}$ is $\mathbf{SK}$, $(b^{i}_j)_{i\in I,j\in J}$ is an $\times$-functional family and $f,g:I\rightarrow J$ are such that 
$$\bigwedge_{i\in I}(a_i\rightarrow b^i_{f(i)})\in \Sigma\textrm{ and }\bigwedge_{i\in I}(a_i\rightarrow b^i_{g(i)})\in \Sigma$$
then $f=g$.
But from the assumption above we get
$$\bigwedge_{i\in I}(a_i\rightarrow b^{i}_{f(i)}\times b^{i}_{g(i)})\in \Sigma$$
from which it follows by the $\times$-functionality  that 
$$\bigwedge_{i\in I}(a_i\rightarrow \delta_{J}(f(i),g(i)))\in \Sigma.$$
This means that 
$$\bigwedge_{i\in I}(a_i\rightarrow \bigexists_{j= f(i)=g(i)}\top)\in \Sigma$$
By the hypothesis of supercompactness we get that $f(i)=g(i)$ for every $i\in I$. Thus $f=g$. 

\end{proof}

From Propositions \ref{sk->fsk} it immediately that:

\begin{corollary} For a family  $(a_{i})_{i\in I}$ of elements of $\mathbb{A}$ we have that   $\mathbf{U\mbox{-}SK}\Rightarrow\mathbf{U\mbox{-}fSK}$.
\end{corollary}
Moreover, trivially one has $\mathbf{cSK}\Rightarrow \mathbf{cInd}$.


\begin{remark}\label{cInd->nmin}
Notice that, by \Cref{nminfSK}, we have that if $(a_i)_{i\in I}$ is $\mathbf{cInd}$, then $a_i\not\equiv_{\Sigma}\bot$ for every $i\in I$.
\end{remark}

\begin{proposition}\label{nminwfsk} If $(a_i)_{i\in I}$ is $\mathbf{wfSK}$ and $a_i\not\equiv_{\Sigma}\bot$ for every $i\in I$, then $(a_i)_{i\in I}$ is $\mathbf{fSK}$.
\end{proposition}
\begin{proof}
If $(a_{i})_{i\in I}$ is $\mathbf{wfSK}$ but not $\mathbf{fSK}$, then there exists $((b^{i}_{j})_{j\in J_i})_{i\in I}$ $\times$-functional and two distinct functions $f,g\in (\Pi i\in I)J_i$ satisfying 
$$\bigwedge_{i\in I}(a_i\rightarrow b^i_{f(i)})\in \Sigma\textrm{ and }\bigwedge_{i\in I}(a_i\rightarrow b^i_{g(i)})\in \Sigma.$$
Then
$$\bigwedge_{i\in I}(a_i\rightarrow b^i_{f(i)}\times b^{i}_{g(i)})\in \Sigma$$
and using the $\times$-functionality we get

$$\bigwedge_{i\in I}(a_i\rightarrow \delta_{J_i}(f(i),g(i)))\in \Sigma$$
If $\overline{i}\in I$ is such that $f(\overline{i})\neq g(\overline{i})$, then
$a_{\overline{i}}\rightarrow (\top\rightarrow \bot)\in \Sigma$. 
From this it follows that $a_{\overline{i}}\equiv_{\Sigma}\bot$.


\end{proof}

\begin{corollary}\label{cor_UwfSK_iff_UfSK}
A family $(a_i)_{i\in I}$ is $\mathbf{U\mbox{-}wfSK}$ if and only if it is  $\mathbf{U\mbox{-}fSK}$.
\end{corollary}
\begin{proof} We already know that by definition $\mathbf{U}${\bf-}$\mathbf{fSK}$ implies $\mathbf{U}${\bf-}$\mathbf{wfSK}$ (see \Cref{fsk->wfsk}). 
Assume now a family $(a_i)_{i\in I}$ to be $\mathbf{U\mbox{-}wfSK}$. Then, combining the last point of \Cref{u->c} with \Cref{cInd->nmin}, we get that $a_i\not \equiv_\Sigma \bot$ for every $i\in I$.
By definition every family $(a_{f(j)})_{j\in J}$ with $f:J\rightarrow I$ is $\mathbf{wfSK}$. Using Proposition \ref{nminwfsk} one gets that each one of these families is $\mathbf{fSK}$. Thus $(a_i)_{i\in I}$ is $\mathbf{U\mbox{-}fSK}$.

\end{proof}

We can summarize the relation between the different properties of families in the general case as follows:

$$
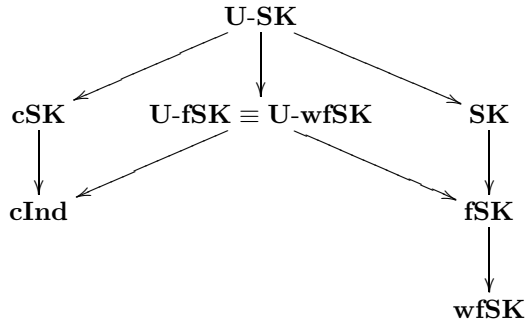
\xymatrix{
    &\mathbf{U\mbox{-}SK} \ar[ld]\ar[d]\ar[rd]  &\\
\mathbf{cSK} \ar[d]   &\mathbf{U\mbox{-}fSK}\equiv \mathbf{U\mbox{-}wfSK}\ar[ld]\ar[rd]    &\mathbf{SK}\ar[d]\\
\mathbf{cInd}    &    &\mathbf{fSK}\ar[d]\\
    &    &\mathbf{wfSK}\\
}
$$

\subsection{Supercompactness and indecomposability in particular classes of implicative algebras}

In this section we analyze the notions of supercompact and indecomposable elements in some particular cases. We start by considering the case of an implicative algebra \emph{compatible with joins}.
\subsubsection*{Compatibility with joins}

\begin{definition}\label{def_IA_comp_joins}
An implicative algebra $\mathbb{A}=(\mathcal{A},\leq,\to,\Sigma)$ is \textbf{compatible with joins} if for every family $(a_i)_{i\in I}$ of its elements and every $b\in \mathcal{A}$ we have that
$$\bigwedge_{i\in I}(a_i\rightarrow b)=\bigvee_{i\in I}a_i\rightarrow b.$$
\end{definition}
For implicative algebras compatible with joins we have the following useful properties, \cite[Prop. 3.32]{miquel_2020}:
\begin{enumerate}
\item $\bot\rightarrow a=\top$
\item $a\times \bot=\bot\times a=\top\rightarrow \bot$
\end{enumerate}
Moreover, using the calculus in \cite{miquel_2020}
one can easily prove that
\begin{lemma}\label{existentialcwj} If $\mathcal{F}$ is the class of set-indexed families of elements of an implicative algebra $\mathbb{A}$ which is compatible with joins, then
$$\bigwedge_{(b_i)_{i\in I}\in \mathcal{F}}\left(\bigvee_{i\in I}b_i\rightarrow \bigexists_{i\in I}b_i\right)\in \Sigma$$
\end{lemma}

\begin{remark}\label{rem_IA_comp_joins_exists_and_V}
    If we consider this property in combination with \Cref{existentialgen} we get that we can substitute $\bigexists$ with $\bigvee$ in logical calculations when we are dealing with an implicative algebra compatible with joins, as shown in \cite[p.490]{miquel_2020}.
\end{remark}

\begin{lemma}\label{equalitycwj}
If $\mathbb{A}$ is compatible with joins, then 
$$\bigwedge_{J\,set}\bigwedge_{j,j'\in J}\left(d_J(j,j')\rightarrow \delta_{J}(j,j')\right)\in \Sigma$$
\end{lemma}
\begin{proof}
$$\bigwedge_{J\,set}\bigwedge_{j,j'\in J}\left(d_J(j,j')\rightarrow \delta_{J}(j,j')\right)=(\top\rightarrow \bigwedge_{c\in \mathcal{A}}((\top\rightarrow c)\rightarrow c))\wedge (\bot\rightarrow (\top\rightarrow \bot))=$$
$$(\top\rightarrow \bigwedge_{c\in \mathcal{A}}((\top\rightarrow c)\rightarrow c))$$
and this can be easily shown to be in $\Sigma$ using the calculus.
\end{proof}
If we consider this property in combination with Proposition \ref{equalitygen} we get that we can substitute $\delta_{J}$ with $d_J$ in logical calculations when we are dealing with an implicative algebra compatible with joins.
\begin{proposition}
Let $\mathbb{A}$ be compatible with joins.
If $(a_i)_{i\in I}$ is $\mathbf{SK}$, then it is $\mathbf{cSK}$.
\end{proposition}
\begin{proof}
Recall from \Cref{rem_IA_comp_joins_exists_and_V} that in an implicative algebra compatible with joins we can substitute $\bigexists$ with $\bigvee$. Now assume that 
$$a_{\overline{i}}\rightarrow \bigexists_{j\in J}b_j\in \Sigma$$
for some $\overline{i}\in I$ and some family $(b_j)_{j\in J}$ of elements of $\mathcal{A}$. 
Let us now consider a family of families
$((c^i_j)_{j\in J_i})_{i\in I}$ of elements of $\mathcal{A}$ defined as follows:
\begin{enumerate}
\item $J_{\overline{i}}=J$ and $J_i=\{\star\}$ if $i\neq \overline{i}$;
\item $c^{\overline{i}}_j:=b_j$, while $c^i_{\star}=\top$.
\end{enumerate}
Since $\mathbb{A}$ is compatible with joins we have that 
$$\bigwedge_{i\in I}(a_{i}\rightarrow \bigexists_{j\in J_i}c^i_j)\in \Sigma\iff\bigwedge_{i\in I}(a_{i}\rightarrow \bigvee_{j\in J_i}c^i_j)\in \Sigma.$$
But  
$$\bigwedge_{i\in I}(a_{i}\rightarrow \bigvee_{j\in J_i}c^i_j)=a_{\overline{i}}\rightarrow \bigvee_{j\in J}b_j $$
and this is an element of $\Sigma$ because $\mathbb{A}$ is compatible with joins. 
Since  $(a_{i})_{i\in I}$ is supercompact by hypothesis we can conclude that there exists a function $f\in (\Pi i\in I)J_i$ such that 
$$\bigwedge_{i\in I}(a_{i}\rightarrow c^{i}_{f(i)})\in \Sigma$$
From this it follows that $f(\overline{i})\in J$ and $a_{\overline{i}}\rightarrow b_{f(\overline{i})}\in \Sigma$. Thus every $a_{\overline{i}}$ is supercompact, i.e  $(a_i)_{i\in I}$ is $\mathbf{cSK}$.
\end{proof}
\begin{proposition}\label{nonbot}Let $\mathbb{A}$ be compatible with joins and $(a_i)_{i\in I}$ be a family of its elements. If  $(a_{i})_{i\in I}$ is $\mathbf{fSK}$, then $a_i\not\equiv_\Sigma \bot$ for every $i\in I$.
\end{proposition}
\begin{proof}
By \Cref{fsk->wfsk} we already know that $\mathbf{fSK}\Rightarrow \mathbf{wfSK}$.
%
%
Therefore, to prove the result it is enough to prove that if $(a_{i})_{i\in I}$ is $\mathbf{wfSK}$ and there exists $\overline{i}\in I$ such that $a_{\overline{i}}\equiv_{\Sigma}\bot$, then $(a_{i})_{i\in I}$ is not $\mathbf{fSK}$. 

Consider the family of families $((b^i_j)_{j\in \{0,1\}})_{i\in I}$ where $b^i_0=\bot$ and $b^i_1=\top$.
Then
$$ \bigwedge_{i\in I}(a_i\rightarrow \bigvee_{j\in \{0,1\}}b^i_j)=\bigwedge_{i\in I}(a_{i}\rightarrow \top)=\top\in \Sigma$$
and since $\mathbb{A}$ is compatible with joins $\bigwedge_{i\in I}(a_i\rightarrow \lar{\exists} _{j\in \{0,1\}}b^i_j)\in \Sigma$.

Moreover
$$\bigwedge_{i\in I}\bigwedge_{j,j'\in I}(b^{i}_{j}\times b^{i}_{j'}\rightarrow \delta(j,j'))=\bigwedge_{i\in I}\left(\bigwedge_{j\in I}(b^{i}_{j}\times b^{i}_{j}\rightarrow \delta(j,j'))\wedge \bigwedge_{j\neq j'\in I}(b^{i}_{j}\times b^{i}_{j'}\rightarrow \delta(j,j'))\right)=$$
$$(\bot\times \bot\rightarrow \top)\wedge (\top\times \top\rightarrow \top)\wedge (\bot\times \top\rightarrow \bot)\wedge (\top\times \bot\rightarrow \bot)=(\top\rightarrow \bot)\rightarrow \bot\in \Sigma $$

since in any implicative algebra compatible with joins $a\times \bot=\bot\times a=\top\rightarrow \bot$ and Lemma \ref{equalitycwj} holds.

Consider now two functions $f,g:I\rightarrow \{0,1\}$ where $f$ is the constant with value $1$ while $g(i)=1$ for every $i\in I$, but $\overline{i}$, for which we have $g(\overline{i})=0$.

We have
$$ \bigwedge_{i\in I}(a_i\rightarrow b^i_{f(i)})=\bigwedge_{i\in I}(a_{i}\rightarrow \top)=\top\in \Sigma$$
and 
$$ \bigwedge(a_i\rightarrow b^i_{g(i)})=\bigwedge_{i\neq \overline{i}\in I}(a_{i}\rightarrow \top)\wedge (a_{\overline{i}}\rightarrow \bot)=a_{\overline{i}}\rightarrow \bot\in \Sigma$$
This $(a_i)_{i\in I}$ is not $\mathbf{fSK}$.

\end{proof}
Combining \Cref{nonbot} with \Cref{nminwfsk} we obtain the following corollary:
\begin{corollary}
    Let $\mathbb{A}$ be compatible with joins and $(a_i)_{i\in I}$ be a family of its elements. Then  $(a_{i})_{i\in I}$ is $\mathbf{fSK}$ if and only if $(a_{i})_{i\in I}$ is $\mathbf{wfSK}$ and $a_i\not\equiv_\Sigma \bot$ for every $i\in I$.
\end{corollary}

\begin{proposition}
Let $\mathbb{A}$ be compatible with joins. If  $(a_{i})_{i\in I}$ is $\mathbf{fSK}$ then it is $\mathbf{cInd}$.
\end{proposition}
\begin{proof}
Assume $(a_{i})_{i\in I}$ is $\mathbf{fSK}$, $a_{\overline{i}}\rightarrow \bigexists_{j\in J}b_j\in \Sigma$ (which, since we are assuming compatibility with joins, is equivalent to $a_{\overline{i}}\rightarrow \bigvee_{j\in J}b_j\in \Sigma$) and $\bigwedge_{j,j'\in J}(b_j\times b_{j'}\rightarrow \delta_J(j,j'))\in \Sigma$.

As a consequence of \Cref{nonbot}, $J\neq \emptyset$. Thus, if for every $i\in I$ and $j\in J$ we define
$$c^i_j=\begin{cases}
b_j\textrm{ if }i=\overline{i}\\
\bot \textrm{ if }i\neq \overline{i},j\neq \overline{j}\\
\top\textrm{ if }i\neq\overline{i},j=\overline{j}\\
\end{cases}$$
where $\overline{j}$ is a fixed element of $J$, we get that 
$$\bigwedge_{i\in I}(a_i\rightarrow \bigvee_{j\in J}c^i_j)=a_{\overline{i}}\rightarrow \bigvee_{j\in J}b_j\in \Sigma$$
from which it follows by compatibility with joins that 
$$\bigwedge_{i\in I}(a_i\rightarrow \bigexists_{j\in J}c^i_j)\in \Sigma$$
Moreover using Lemma \ref{equalitycwj} we get that 
$$\bigwedge_{i\in I}\bigwedge_{j,j'\in J}(c^i_j\times c^i_{j'}\rightarrow \delta_J(j,j'))\equiv_{\Sigma}\bigwedge_{j,j'\in J}(b_j\times b_{j'}\rightarrow d_J(j,j'))\wedge \bigwedge_{i\neq \overline{i}} ((\top\rightarrow \bot)\rightarrow \bot)=$$
$$\bigwedge_{j\neq j'\in J}(b_j\times b_{j'}\rightarrow \bot)\wedge \bigwedge_{i\neq \overline{i}} ((\top\rightarrow \bot)\rightarrow \bot) =$$
$$\bigwedge_{j\neq j'\in J}(b_j\times b_{j'}\rightarrow \bot)\equiv_{\Sigma}\bigwedge_{j,j'\in J}(b_j\times b_{j'}\rightarrow \delta_J(j,j'))\in \Sigma$$
since $\delta_{J}(j,j)=\top$, $x\rightarrow \top=\top$ and $x\times y\geq \top\rightarrow \bot$ for every $x,y$ in $\mathbb{A}$.

Thus, since $(a_{i})_{i\in I}$ is $\mathbf{fSK}$, we get the existence of a function $f:I\rightarrow J$ such that 
$$\bigwedge_{i\in I}(a_{i}\rightarrow c^i_{f(i)})\in \Sigma.$$
In particular $a_{\overline{i}}\rightarrow b_{f(\overline{i})}\in \Sigma$ and we can conclude.
\end{proof}
Thus if we add the assumption that $\mathbb{A}$ is compatible with joins the situation can be summarized as follows:

$$
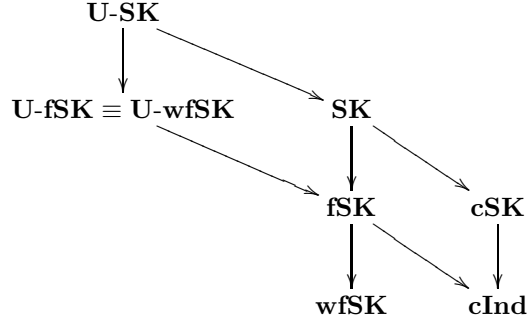
\xymatrix{
\mathbf{U\mbox{-}SK}\ar[rd]\ar[d] &\\
\mathbf{U\mbox{-}fSK}\equiv \mathbf{U\mbox{-}wfSK}\ar[rd] &\mathbf{SK}\ar[d]\ar[rd]\\
       &\mathbf{fSK}\ar[d]\ar[rd]  &\mathbf{cSK}\ar[d]\\
                &  \mathbf{wfSK}        &\mathbf{cInd}\\
}$$

\subsubsection*{$\Sigma$ closed under $\bigwedge$}
We start by recalling that, in general, the separator of an implicative algebra is not required to be closed under arbitrary infima, even if there are several situations in which this is the case, e.g.\ for an implicative algebra associated with a complete Heyting algebra. One can also notice that a separator $\Sigma$ is closed under arbitrary infima if and only if it is a principal filter. In \cite[Prop.4.13]{miquel_2020} implicative algebras of which the separator is a principal ultrafilter are characterized as those giving rise to a tripos which is isomorphic to a forcing tripos.
\begin{proposition}\label{fefinf}
If $\Sigma$ is closed under arbitrary infima $\bigwedge$, then all $\mathbf{cSK}$ families are $\mathbf{U\mbox{-}SK}$. 
\end{proposition}
\begin{proof}
Let $(a_i)_{i\in I}$ be a $\mathbf{cSK}$ family.
We have to prove that 
for every $f:J\rightarrow I$, the family $(a_{f(j)})_{j\in J}$ is $\mathbf{SK}$.

Let us assume hence that 
$$\bigwedge_{j\in J}\left(a_{f(j)}\rightarrow \bigexists_{k\in K_j}b^j_k\right)\in \Sigma$$
Using the supercompactness of every component $a_i$ and the axiom of choice, we get the existence of a function $g:(\Pi j\in J)K_j$ such that $a_{f(j)}\rightarrow b^j_{g(j)}\in \Sigma$ for every $j\in J$. Using the hypothesis of closure of $\Sigma$ under $\bigwedge$ we get
$$\bigwedge_{j\in J}\left(a_{f(j)}\rightarrow b^j_{g(j)}\right)\in \Sigma$$
which concludes the proof.


\end{proof}
Similarly one proves that
\begin{proposition}If $\Sigma$ is closed under arbitrary infima $\bigwedge$, then all $\mathbf{cInd}$ families are $\mathbf{U\mbox{-}fSK}$.
\end{proposition}

Summarizing, in the case the separator of an implicative algebra is closed under arbitrary infima we have the following situation:

    $$\xymatrix{
&\mathbf{U\mbox{-}SK}\equiv  \mathbf{cSK}\ar[d]\ar[rd]\\
&\mathbf{U\mbox{-}fSK}\equiv \mathbf{U\mbox{-}wfSK}\equiv  \mathbf{cInd}\ar[rd]  &\mathbf{SK}\ar[d]\\
& &\mathbf{fSK}\ar[d]  \\
    & & \mathbf{wfSK}\\
}$$

\subsubsection*{$\bigexists$-distributivity}
\begin{definition} An implicative algebra is {\bf  $\bigexists$-distributive} if 
$$\bigwedge_{((b^k_j)_{j\in J_k})_{k\in K}\in \mathcal{F}}\left(\bigwedge_{k\in K}\bigexists_{j\in J_k}b^k_{j}\rightarrow \bigexists_{f\in (\Pi k\in K)J_k}\bigwedge_{k\in K}b^{k}_{f(k)}\right)\in \Sigma$$
where $\mathcal{F}$ is the class of all families of families $((b^k_j)_{j\in J_K})_{k\in K}$ of elements of $\mathcal{A}$.
\end{definition}
\begin{remark}Although $\mathcal{F}$ in the previous definition is a class, the infimum there is well-defined since it coincides with the infimum of a subset of $\mathcal{A}$. 
\end{remark}
\begin{proposition}\label{distcomp} If $\mathbb{A}$ is  $\bigexists$-distributive, then every $\mathbf{SK}$ family is $\mathbf{U\mbox{-}SK}$. 
Moreover, every $\mathbf{fSK}$ family is $\mathbf{U\mbox{-}fSK}$.
\end{proposition}
\begin{proof}
We use Proposition \ref{carusk}.
Thus let $(a_i)_{i\in I}$ be a supercompact family and consider a family of families of families $(((b^k_j)_{j\in J_k})_{k\in K_i})_{i\in I}$ of elements of $\mathcal{A}$, with $(K_i)_{i\in I}$ a family of pairwise disjoint sets, such that
$$\bigwedge_{i\in I}(a_i\rightarrow \bigwedge_{k\in K_i}\bigexists_{j\in J_k}b^k_j)\in \Sigma$$
By uniform $\bigexists$-distributivity we get that 
$$\bigwedge_{i\in I}(a_i\rightarrow \bigexists_{g\in (\Pi k\in K_i)J_k}\bigwedge_{k\in K_i}b^k_{g(k)})\in \Sigma$$
Using supercompactness of $(a_i)_{i\in I}$ we hence get that there exists $f\in (\Pi i\in I)(\Pi k\in K_i)J_k$ such that
$$\bigwedge_{i\in I}(a_i\rightarrow \bigwedge_{k\in K_i}b^k_{f(i)(k)})\in \Sigma$$
Since the $K_i$'s are pairwise disjoint, the function $\widetilde{f}$ sending every $k\in K_i$ to $f(i)(k)$ is well defined and $\widetilde{f}(k)\in J_k$; hence we get 
$$\bigwedge_{i\in I}(a_i\rightarrow \bigwedge_{k\in K_i}b^k_{\widetilde{f}(k)})\in \Sigma$$
Thus we get uniform supercompactness of $(a_{i})_{i\in I}$.

The proof of the second part of the statement is analogous.
\end{proof}
Thus if we add the assumption that $\mathbb{A}$ is $\lar{\exists}$-distributive, the situation can be summarized as follows:
    $$
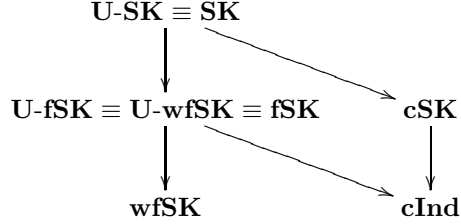
\xymatrix{
&\mathbf{U\mbox{-}SK}\equiv  \mathbf{SK}\ar[d]\ar[rd]\\
&\mathbf{U\mbox{-}fSK}\equiv \mathbf{U\mbox{-}wfSK}\equiv  \mathbf{fSK}\ar[rd]\ar[d]  &\mathbf{cSK}\ar[d]\\
&\mathbf{wfSK} &\mathbf{cInd}  \\
}$$


\subsection{Supercompactness and indecomposability in complete Heyting algebras}\label{sect_supercompact_in_heyting_and_boolean_alg}
In the case of a complete Heyting algebra, the notion of supercompact element introduced in \Cref{def_supercompact_element_IA} coincides with the ordinary notion, i.e an element $a$ is supercompact if $a\leq \bigvee_{i\in I}b_i$ implies the existence of an  $\overline{i}\in I$ such that $a\leq b_{\overline{i}}$, for every set-indexed family $(b_i)_{i\in I}$ of elements of $\mathcal{A}$, that is if $a$ is a  supercompact element in the localic sense (see \cite{pp}). In particular, in the boolean case supercompact elements are exactly atoms. Indecomposable elements in Heyting algebras express just a notion of connectedness which is called in fact \emph{indecomposability} and it can be shown to be equivalent to usual connectedness\footnote{An element $a$ of a Heyting algebra is connected if and only if $a\leq b\vee c$ implies $a\leq b$ or $a\leq c$.} plus the fact of being different from $\bot$ (see e.g.\ \cite{OC}). For complete Boolean algebras, supercompactness coincides with indecomposability. Indeed, for every family of elements $(b_i)_{i\in I}$ of a Boolean algebra one can construct a new family $(\widetilde{b}_i)_{i\in I}$ satisfying the following properties:
\begin{enumerate}
\item $\widetilde{b}_{i}\leq b_i$ for every $i\in I$;
\item $\widetilde{b}_i\wedge \widetilde{b}_j=\bot$ for every $i\neq j$ in $I$
\item $\bigvee_{i\in I} \widetilde{b}_i=\bigvee_{i\in I}b_i$.
\end{enumerate}
Notice that in general indecomposable elements of a Heyting algebra are not supercompact. For example, we can consider the complete Heyting algebra of open subsets of reals: it is immediate to check that there are no supercompact open subsets of reals, while each open interval is indecomposable.

Every Heyting algebra can be easily seen to be compatible with joins. Moreover $\Sigma$, being $\{\top\}$, is closed under $\bigwedge$. Thus we get that in there for families we have $\mathbf{U\mbox{-}SK}\equiv \mathbf{SK}\equiv \mathbf{cSK}$. Families satisfying these properties are exactly componentwise supercompact (in localic sense) families. 
Moreover, $\mathbf{U\mbox{-}fSK}\equiv \mathbf{U\mbox{-}wfSK}\equiv \mathbf{fSK}\equiv \mathbf{cInd}$. Families satisfying these properties are exactly 
componentwise indecomposable families.

Finally, one can easily see that a family $(a_i)_{i\in I}$ is $\mathbf{wfSK}$ if and only if one of the following happens:
\begin{enumerate}
\item $I=\emptyset$ or
\item $I\neq \emptyset$, $a_i$ is connected for every $i\in I$, but at least one of the $a_i$'s is indecomposable.
\end{enumerate}

$$\xymatrix{
&\mathbf{U\mbox{-}SK}\equiv \mathbf{SK}\equiv \mathbf{cSK}\ar[d]\\
&\mathbf{U\mbox{-}fSK}\equiv \mathbf{U\mbox{-}wfSK}\equiv \mathbf{fSK}\equiv \mathbf{cInd}\ar[d]\\
& \mathbf{wfSK}  \\
}$$





Since for complete Boolean algebras indecomposable elements are exactly supercompact elements, for Boolean algebras we have that
$\mathbf{U\mbox{-}SK}\equiv \mathbf{SK}\equiv \mathbf{cSK}\equiv \mathbf{U\mbox{-}fSK}\equiv \mathbf{U\mbox{-}wfSK}\equiv\mathbf{fSK}\equiv \mathbf{cInd}$
and these are exactly componentwise atomic families; $\mathbf{wfSK}$ families are families whose components are atoms or minima, but not all minima if the family is non-empty.

Notice that a complete Heyting algebra is $\bigexists$-distributive if and only if it is completely distributive. For complete Boolean algebras this amounts to the requirement of being complete atomic (that is of being isomorphic to a powerset algebra).

\subsection{Supercompactness and indecomposability in different kinds of realizability}\label{sect_supercompact_in_realizability}
\subsubsection*{Realizability}
In the case of an implicative algebra $(\mathcal{P}(R),\subseteq, \Rightarrow, \mathcal{P}(R)\setminus \{\emptyset\})$ coming from a combinatory algebra $(R,\cdot)$, supercompact elements are non-empty subsets of $R$, since they are all equivalent to $\top=R$, and since $\bot=\emptyset$ (we are using Remarks \ref{supmin} and \ref{supmax}). Every non-empty subset of $R$ is also indecomposable (since it is supercompact) and every indecomposable subset is non-empty. Thus, supercompactness and indecomposability coincide for elements.

We can also notice that the  implicative algebra arising from a combinatory algebra $(R,\cdot)$ is compatible with joins and $\bigexists$-distributive.
The former property can be trivially verified, while to prove the latter, it is enough to observe that since the implicative algebra coming from a combinatory algebra $(R,\cdot)$ is compatible with joins, it is sufficient to prove that 
there is a realizer $r$ independent from the specific family such that 
$$r\in \bigcap_{i\in I}\bigcup_{j\in J_i}B^i_j\Rightarrow \bigcup_{f\in (\Pi i\in I)J_i}\bigcap_{i\in I}B^i_{f(i)}$$
Since the antecedent and the consequent of the implication above are equal sets, $r$ can be taken to be $\mathbf{i}$.

Therefore, by \Cref{distcomp}, we have that $\mathbf{U\mbox{-}SK}$ families coincide with $\mathbf{SK}$ families, and $\mathbf{U\mbox{-}fSK}$ families coincide with $\mathbf{fSK}$ families.
We now show that for every $\mathbf{SK}$ family $(A_i)_{i\in I}$ there exists a function $f:I\rightarrow R$ such that 
$$\bigcap_{i\in I}(A_i\Rightarrow \{f(i)\})\neq \emptyset\textrm{ and }\bigcap_{i\in I}( \{f(i)\}\Rightarrow A_i)\neq \emptyset$$
In order to do this we consider the family of families $((\{j\})_{j\in A_i})_{i\in I}$.
One easily can show that 
$$\lambda x.\lambda y.yx\in \bigcap_{i\in I}(A_i\Rightarrow \bigexists_{j\in A_i}\{j\})$$
Using supercompactness we get the existence of a function $f\in(\Pi i\in I)A_i\subseteq R^I$ and a realizer $s$ such that 
$$s\in \bigcap_{i\in I}(A_i\Rightarrow \{f(i)\})$$
Since $\mathbf{i}\in \bigcap_{i\in I}( \{f(i)\}\Rightarrow A_i)$ we can conclude.

Conversely, we show that if $(A_i)_{i\in I}$ is a family satisfying the property above, then it is $\mathbf{SK}$.
So let us assume 
$$r\in \bigcap_{i\in I}(A_{i}\Rightarrow \bigexists_{j\in J_i}B^i_j)$$
Using compatibility with joins and the hypothesis on $(A_i)_{i\in I}$ we get that there exists a realizer $r'$ such that 
$$r'\in \bigcap_{i\in I}(\{f(i)\}\Rightarrow \bigcup_{j\in J_i}B^i_j)$$
Now, for every $i\in I$ there exists $j\in J_i$ such that $r'\cdot f(i)\in B^i_j$. Using the axiom of choice we hence produce a function $g\in (\Pi i\in I)J_i$ such that 
$$r'\in \bigcap_{i\in I}(\{f(i)\}\Rightarrow B^{i}_{g(i)})$$
Using the hypothesis on the family again we get the existence of a realizer $s$ such that 
$$s\in \bigcap_{i\in I}(A_i\Rightarrow B^{i}_{g(i)})$$

Assume now $(R,\cdot)$ to be a non-trivial combinatory algebra and consider three distinct realizers $r_0,r_1,r_2\in R$ (this is always possible in a non-trivial combinatory algebra) and the family $(A_i)_{i\in \{0,1,2\}}$
where $A_0=\{r_0,r_1\}$, $A_1=\{r_1,r_2\}$ and $r_2=\{r_0,r_2\}$. Let us assume there is a function $f:\{0,1,2\}\rightarrow R$ such that 
$$\bigcap_{i=0,1,2}(A_i\Rightarrow \{f(i)\})\neq \emptyset$$
Then, since $A_i\cap A_j\neq \emptyset$ for every $i,j\in \{0,1,2\}$ we get $f(0)=f(1)=f(2)=\overline{a}$. If we now assume also that there exists $r\in \mathsf{R}$ such that 
$$r\in \bigcap_{i=0,1,2}(\{\overline{a}\}\Rightarrow A_i)$$
then we get that $r\cdot\overline{a}\in A_0\cap A_1\cap A_2$ which is in contradiction with our assumption. We have hence produced an example of a $\mathbf{cSK}$ family which is not $\mathbf{SK}$. Thus for the implicative algebra arising from a non-trivial combinatory algebra $\mathbf{cSK}$ does not imply $\mathbf{SK}$.
Since supercompact elements coincide with functionally supercompact elements (they are just non-empty sets of realizers), we have that 
$\mathbf{cSK}\equiv \mathbf{cInd}$.
Moreover, it is easy to show that in the case of realizability $\mathbf{cSK}$ implies $\mathbf{wfSK}$, from which it follows that $\mathbf{cSK}$ implies $\mathbf{fSK}$. Conversely, $\mathbf{fSK}$ implies $\mathbf{cSK}$ by Proposition \ref{nonbot}. Thus, also $\mathbf{fSK}\equiv \mathbf{cSK}$.

Finally, one can easily prove that a family $(A_i)_{i\in I}$ is $\mathbf{wfSK}$ if and only if $I=\emptyset$ or there exists $\overline{i}\in I$ such that $A_{\overline{i}}\neq \emptyset$.

$$\xymatrix{
\mathbf{U\mbox{-}SK}\equiv \mathbf{SK}\ar[d]\\
 \mathbf{cSK}\equiv\mathbf{U\mbox{-}fSK}\equiv \mathbf{U\mbox{-}wfSK}\equiv \mathbf{fSK}\equiv \mathbf{cInd}\ar[d]\\
\mathbf{wfSK}
}$$
\subsubsection*{Relative realizability}
It is easy to check that the implicative algebras of relative realizability are compatible with joins and $\lar{\exists}$-distributive.

Supercompact elements are exactly those  $A\subseteq R$ which are equivalent to a singleton. Indeed, since the combinator $\mathbf{i}$ can be assumed to be in $R_\#$, the implicative algebra is compatible with joins and $\mathbf{i}\in A\Rightarrow \bigcup_{a\in A}\{a\}$, we get that if $A$ is supercompact, then there exists $a\in A$ and $r\in R_\#$, such that $r\in A\Rightarrow\{a\}$. Since $\mathbf{i}\in \{a\}\Rightarrow A$, we can conclude that $A\equiv_{\Sigma^r_{\mathcal{R},\mathcal{R}_\#}}\{a\}$. Conversely, each such an element is supercompact since singletons are easily shown to be so. In particular, all subsets $A$ such that $A\cap R_\#\neq \emptyset$ are supercompact. Another example of supercompact is a set of the form $P_a:=\{\mathbf{p}ab|\,b\in R\}$ where $a\in R$ is fixed and where $\mathbf{p}$ is the usual pairing combinator which we can assume to be in $R_\#$. Indeed, this subset is equivalent to the singleton $\{a\}$ since $\mathbf{p}_1\in P_a\Rightarrow\{a\}$ and the first-projection combinator can be assumed to be in $R_\#$, while $\lambda x.\mathbf{p}xx\in \{a\}\Rightarrow P_a$ and this combinator too can be assumed to be in $R_\#$.


From the characterization of $\times$-disjoint families in relative realizability implicative algebras, it follows that indecomposable elements in this case are just non-empty subsets of $R$. Thus in general $\mathbf{cInd}$ is not equivalent to $\mathbf{cSK}$ for families. 

Moreover, $\mathbf{SK}$ is not equivalent to $\mathbf{cSK}$ (realizability is a special case of relative realizability). 

Since the implicative algebra is $\bigexists$-distributive, 
then $\mathbf{U\mbox{-}SK}\equiv \mathbf{SK}$ and $\mathbf{U\mbox{-}fSK}\equiv\mathbf{U\mbox{-}wfSK}\equiv \mathbf{fSK}$ and one can easily show, as for realizability, that $\mathbf{SK}$ families are those families equivalent to singleton families, i.e. families of the form $(\{f(i)\})_{i\in I}$ for some function $f:I\rightarrow R$.

One can also show easily that $\mathbf{cInd}$ is equivalent to $\mathbf{fSK}$.
Finally, as in the realizability case, one can easily prove that a family $(A_i)_{i\in I}$ is $\mathbf{wfSK}$ if and only if $I=\emptyset$ or there exists $\overline{i}\in I$ such that $A_{\overline{i}}\neq \emptyset$.
$$\xymatrix{
\mathbf{U\mbox{-}SK}\equiv\mathbf{SK}\ar[d]\\
\mathbf{cSK}\ar[d]\\
\mathbf{U\mbox{-}fSK}\equiv\mathbf{U\mbox{-}wfSK}\equiv\mathbf{fSK}\equiv\mathbf{cInd}\ar[d]\\
\mathbf{wfSK}\\
}$$

\subsubsection*{Nested realizability}
In nested realizability implicative algebras supercompact elements are pairs $(A,B)$ which are equivalent to a pair of the form $(X,\{b\})$ with $X\subseteq \{b\}$. Indeed, one can notice that 
$$\mathbf{i}\in (A\Rightarrow_\# \bigcup_{b\in B}(\{b\}\cap A))\; \cap\; (B\Rightarrow \bigcup_{b\in B}\{b\})$$
Thus, since the implicative algebra of nested realizability is compatible with joins, if $(A,B)$ is supercompact, then there exists $b\in B$ such that $(\{b\}\cap A,\{b\})\equiv_{\Sigma_{\mathcal{R},\mathcal{R}_\#}}(A,B)$. The converse can be easily checked. 

Functionally supercompact elements are those $(A,B)$ with $B\neq \emptyset$.

Nested realizability implicative algebras are also $\bigexists$-distributive. Thus, also in this case we get the equivalence between $\mathbf{SK}$ and $\mathbf{U\mbox{-}SK}$ and in the non-trivial case we get that $\mathbf{cSK}$ does not imply $\mathbf{SK}$ using an argument similar to that used for the realizability case. 

In particular a family $(A_i,B_i)_{i\in I}$ is $\mathbf{SK}$ if and only if there exists a family $(C_i,D_i)_{i\in I}$ with $D_i$ singleton for every $i\in I$ such that 
$$\bigcup_{i\in I}\left(A_i\Rightarrow_\# C_i\cap B_i\Rightarrow D_i\right)\neq \emptyset\textrm{ and }\bigcup_{i\in I}\left(C_i\Rightarrow_\# A_i\cap D_i\Rightarrow B_i\right)\neq \emptyset$$
In order to prove this one just has to work with the family of families $((\{j\}\cap A_i\},\{j\})_{j\in B_i})_{i\in I}$.

We can characterize $\mathbf{fSK}$ families as those families $(A_i,B_i)_{i\in I}$ where $B_i\neq \emptyset$ for every $i\in I$, that is $\mathbf{fSK}\equiv \mathbf{cInd}$. This is also equivalent to $\mathbf{U\mbox{-}fSK}$ by $\lar{\exists}$-distributivity.

\subsubsection*{Modified realizability}

Supercompact elements in modified realizability implicative algebras are pairs $(A,B)$ with $A\neq\emptyset$; indeed one can exploit the fact that in general $\bigexists_{i\in I}b_{i}\rightarrow \bigvee_{i\in I}b_i\in \Sigma$ (see Proposition \ref{existentialgen}), to prove that every $(A,B)$ with $A\neq\emptyset$ is supercompact, while one can easily show that $(\emptyset,B)\Rightarrow_n \bigexists \emptyset\in \Sigma$ from which it follows that $(\emptyset,B)$ can never be supercompact. From this characterization of supercompact elements it follows also that supercompact elements coincide with indecomposable elements.

By considering the empty family, one can easily check that a modified realizability implicative algebra is not compatible with joins, provided $\mathcal{R}$ is not trivial. This fact also makes difficult to prove whether these implicative algebras are $\lar{\exists}$-distributive or not. We leave as an open problem for further investigations to characterize the different kinds of families in this case.

\subsection{Supercoherent implicative algebras}

\begin{definition} An implicative algebra $\mathbb{A}$ satisfies {\bf choice rule} if for every family $((b^k_j)_{j\in J_K})_{k\in K}$ of families of elements of $\mathcal{A}$, if 
$\bigwedge_{k\in K}\bigexists_{j\in J_k}b^k_j\in \Sigma$, 
then there exists $f\in (\Pi k\in K)J_k$ such that
$\bigwedge_{k\in K}b^k_{f(k)}\in \Sigma$.
\end{definition}

It is an easy exercise to prove that the following holds.
\begin{proposition}\label{prop_RC_implies_every_top_is_unif_supercompact}
$\mathbb{A}$ satisfies the choice rule if and only if $\top_I$ is $\mathbf{U\mbox{-}SK}$ for every $I$.
\end{proposition}
\begin{example}
The implicative algebras of realizability, nested realizability and relative realizability satisfy the choice rule. A complete Heyting algebra satisfies choice rule if and only if it is supercompact (see \cite{pp}).
\end{example}
Motivated by the characterization presented in \cite{MaiettiTrotta21} and by the realizability and localic examples, we introduce the following definition.

\begin{definition}\label{def_unif_supercoherent}
    An implicative algebra $\mathbb{A}$ is said to be \textbf{uniformly supercoherent} if:
    \begin{itemize}
        \item it satisfies the choice rule;
        \item if $(a_i)_{i\in I}$ and $(b_i)_{i\in I}$ are $\mathbf{U\mbox{-}SK}$ families, then $(a_i\times b_i)_{i\in I}$ is a $\mathbf{U\mbox{-}SK}$ family;
        \item for every family $(a_j)_{j\in J}$ there exists a set $I$, a function $\freccia{I}{f}{J}$ and a $\mathbf{U\mbox{-}SK}$ family $(b_i)_{i\in I}$ such that 
        $$\bigwedge_{j\in J}(a_j\rightarrow \bigexists_{f(i)=j} b_i) \in \Sigma\qquad\textrm{ and }\qquad\bigwedge_{j\in J}((\bigexists_{f(i)=j} b_i)\rightarrow a_j)\in \Sigma$$
    \end{itemize}
\end{definition}
\begin{remark}
We anticipate here that, by definition, an implicative algebra is uniformly supercoherent if and only if its implicative tripos is an instance of the full existential completion, see \cite[Thm. 4.16 and Thm. 7.32]{MaiettiTrotta21}. We will see more details about this in the next section and in particular in \Cref{rem_SK_p_are_FEF_in_realizability}.
\end{remark}
\begin{example} \label{ex_superco_realizability_and Heyting_alg}
An implicative algebra coming from a complete Heyting algebra is uniformly supercoherent if and only if the corresponding locale is supercoherent. The implicative algebras of realizability w.r.t.\ a combinatory algebra are uniformly supercoherent. We refer to \cite{MaiettiTrotta21} for all the details.
\end{example}

\begin{definition}\label{def_unif_functional_supercoherent}
    An implicative algebra $\mathbb{A}$ is said to be \textbf{uniformly functional-supercoherent} if:
    \begin{itemize}
        \item it satisfies the choice rule;
        \item if $(a_i)_{i\in I}$ and $(b_i)_{i\in I}$ are $\mathbf{U\mbox{-}SK}$ families, then $(a_i\times b_i)_{i\in I}$ is a $\mathbf{U\mbox{-}SK}$ family;
        \item for every $\mathbf{U\mbox{-}fSK}$ family $(a_j)_{j\in J}$ there exists a set $I$, a function $\freccia{I}{f}{J}$ and a $\mathbf{U\mbox{-}SK}$ family $(b_i)_{i\in I}$ such that 
        $$\bigwedge_{j\in J}(a_j\rightarrow \bigexists_{f(i)=j} b_i) \in \Sigma\qquad\textrm{ and }\qquad\bigwedge_{j\in J}((\bigexists_{f(i)=j} b_i)\rightarrow a_j)\in \Sigma$$
    \end{itemize}
\end{definition}

\begin{remark}\label{rem_SK_iff_U_SK_in_unif_suprc_alg}
    Notice that in a uniformly supercoherent implicative algebra, we have that $\mathbf{SK}\equiv\mathbf{U\mbox{-}SK}$. Again, this can be considered as a particular case of the \cite[Lem. 4.11]{MaiettiTrotta21}.
\end{remark}

\subsection{Modest and core families}
We introduce here some notions which we will use later.
\begin{definition} A family $(a_i)_{i\in I}$ of elements of $\mathbb{A}$ is
\begin{enumerate}
\item $\wedge$-\emph{modest} if it is $\mathbf{U\mbox{-}fSK}$ and $\wedge$-disjoint;
\item $\times$-\emph{modest} if it is $\mathbf{U\mbox{-}fSK}$ and $\times$-disjoint;
\item a $\wedge$-\emph{core family} if it is $\mathbf{U\mbox{-}SK}$ and $\wedge$-disjoint.
\item a $\times$-\emph{core family} if it is $\mathbf{U\mbox{-}SK}$ and $\times$-disjoint.
\end{enumerate}
\end{definition}
Using results and definitions in the previous section one has that:
$$\xymatrix{
\times\textrm{-core}\ar@2[r]\ar@2[d]   &\wedge\textrm{-core}\ar@2[d]\\
 \times\textrm{-modest}\ar@2[r]   &\wedge\textrm{-modest}\\                       
}$$
\begin{example}
In a complete Heyting algebra, since $\times=\wedge$, $\wedge$-modest families coincide with $\times$-modest families: they are families whose elements are pairwise disjoint and indecomposable. Also $\wedge$-core families coincide with $\times$-core families: they are families of pairwise disjoint supercompact elements.
\end{example}
\begin{example}
In the (total) realizability case $\wedge$-modest families are families of pairwise disjoint non-empty sets of realizers, that is modest sets or PERs (see \cite{rosMS}). 
A family $(A_i)_{i\in I}$ is a $\wedge$-core family if and only if it is equivalent to a family of the form $(\{f(i)\})_{i\in I}$ with $f:I\rightarrow A$ injective.
Lastly, a family $(A_i)_{i\in I}$ is $\times$-modest if and only if it is a $\times$-core family if and only if $I=\emptyset$ or $I$ is a singleton $\{\star\}$ and $A_\star$ is non-empty.
\end{example}

\section{Triposes from implicative algebras}
In \cite{miquel_2020} Miquel introduced the notion of \emph{tripos} associated with an implicative algebra, and he proved in \cite{miquel_2020_2} that every $\set$-based tripos, i.e.\ every tripos as originally introduced in \cite{TT}, is equivalent to a tripos arising from an implicative algebra. In this section, we recall the definition of $\set$\emph{-based tripos} (from \cite{TT}), \emph{implicative tripos} and the main result of Miquel. 

\textbf{Notation:} we denote by $\mathsf{Hey}$ the category of Heyting algebras and their morphisms, and we denote by $\mathsf{Pos}$ the category of posets and their morphisms.

\begin{definition}[tripos]
    A ($\set$-based) \textbf{tripos} is a functor $\hyperdoctrine{\set}{\sP}$ such that
    \begin{itemize}
        \item for every function $\freccia{X}{f}{Y}$ the re-indexing functor $\freccia{\sP (Y)}{\sP_f}{\sP(X)}$ has a left adjoint $\freccia{\sP(X)}{\exists_f}{\sP(Y)} $ and a right adjoint $\freccia{\sP(X)}{\forall_f}{\sP(Y)} $ in the category $\mathsf{Pos}$, satisfying the Beck-Chevalley condition (BCC), i.e. for every pullback
\[\begin{tikzcd}
	W & Z \\
	X & Y
	\arrow["{f'}", from=1-1, to=1-2]
	\arrow["g", from=1-2, to=2-2]
	\arrow["{g'}"', from=1-1, to=2-1]
	\arrow["f"', from=2-1, to=2-2]
	\arrow["\lrcorner"{anchor=center, pos=0.125}, draw=none, from=1-1, to=2-2]
\end{tikzcd}\]
we have that $\sP_g\exists_f=\exists_{f'}\sP_{g'}$ and $\sP_g\forall_f=\forall_{f'}\sP_{g'}$.
\item there exists a \emph{generic predicate}, namely there exists a set $\Sigma$ and an element $\sigma$ of $\sP(\Sigma)$ such that for every element $\alpha$ of $\sP(X)$ there exists a function $\freccia{X}{f}{\Sigma}$ such that $\alpha=\sP_f(\sigma)$;
    \end{itemize}
\end{definition}
\begin{remark}[Frobenius reciprocity]
    Employing the preservation of the Heyting implication $\to$ by $\sP_f$, it is straightforward to check that every tripos $\sP$ satisfies the so-called \emph{Frobenius reciprocity} (FR), namely:
    \[ \exists_f(\sP_f(\alpha)\wedge \beta)= \alpha \wedge \exists_f (\beta)\mbox{ and }  \forall_f(\sP_f(\alpha)\rightarrow \beta)= \alpha \rightarrow \forall_f (\beta)\]
    for every  function $\freccia{X}{f}{Y}$, $\alpha$ in $\sP(Y)$ and $\beta$ in $\sP (X)$. See \cite[Rem. 1.3]{TT}.
\end{remark}
Given a $\set$-based tripos $\hyperdoctrine{\set}{\sP}$, we will denote by $\delta_X:=\exists_{\Delta_X}(\top)$ the so-called \emph{equality predicate} on $X$ of the tripos.

Now let us consider an implicative algebra $\mathbb{A}=(\mathcal{A},\leq, \to,\Sigma)$. For each set $I$ we can define a new implicative algebra $(\mathcal{A}^I,\leq^I, \to^I,\Sigma[I])$ where $\mathcal{A}^I$ denotes the set of functions from $I$ to $\mathcal{A}$ (which we call \emph{predicates} over $I$), $\leq^I$ is the point-wise order ($f\leq^I g$ if and only if $f(i)\leq g(i)$ for every $i\in I$), for every $f,g\in \mathcal{A}^I$ and $i\in I$, the function $f\to^I g$ is defined by $(f\to^I g)(i):=f(i)\to g(i)$, and 
$\Sigma[I]\subseteq \mathcal{A}^I$ is the so-called \emph{uniform power separator}  defined as:
\[ \Sigma[I]:=\{f\in A^I |\exists s\in \Sigma, \forall i\in I, s\leq f(i)\}=\{f\in A^I|\,\bigwedge_{i\in I} f(i)\in \Sigma\}.\]

As we have already seen, given an implicative algebra $(\mathcal{A},\leq,\to,\Sigma)$, we have an induced binary relation of \emph{entailment} on $\mathcal{A}$, written $a\vdash_\Sigma b$ and defined by
\[a \vdash_\Sigma b \iff (a\to b)\in \Sigma.\]
It is direct to check that this binary relation gives a preorder $(\mathcal{A},\vdash_\Sigma)$ on $\mathcal{A}$. 

In~\cite[Sec. 4]{miquel_2020} it is shown that each implicative algebra $(\mathcal{A},\leq,\to,\Sigma)$ induces a tripos $\hyperdoctrine{\set}{\mathsf{P}}$ defined as follows:
\begin{definition}[implicative tripos]\label{def:implicative tripos}
   Let $(\mathcal{A},\leq,\to, \Sigma)$ be an implicative algebra. For each set $I$ the Heyting algebra $\mathsf{P}(I)$ is given by the posetal reflection of the preorder $(\mathcal{A}^I, \vdash_{\Sigma [I]})$. For each function $\freccia{I}{f}{X}$, the functor $\freccia{\mathsf{P}(X)}{\mathsf{P}_f}{\mathsf{P}(I)}$ acts by precomposition, that is $P_f([g]):=[g\circ f]$ for every $f:I\rightarrow \mathcal{A}$. 
\end{definition}
The functor defined in Definition~\ref{def:implicative tripos} can be proved to be a $\set$-based tripos, see~\cite[Sec. 4]{miquel_2020}, and it is called \emph{implicative tripos}.

In~\cite[Thm. 1.1]{miquel_2020_2} Miquel proved that the notion of implicative tripos is general enough to encompass all $\set$-based triposes. In particular we have the following result:
\begin{theorem}\label{thm:every set base tripos is an implicative tripos}
   Every $\set$-based tripos is isomorphic to an implicative tripos.
\end{theorem}

\begin{example}[realizability tripos]
    The realizability tripos introduced in \cite{TT} corresponds to the implicative tripos arising from the implicative algebra given by a partial combinatory algebra \Cref{subse_examples_IA}.
\end{example}
\begin{example}[localic tripos]
    The localic tripos introduced in \cite{TT} corresponds to the implicative tripos arising from the implicative algebra given by a complete Heyting algebra \Cref{subse_examples_IA}.
\end{example}

\subsection{Supercompact predicates of implicative triposes}
In this section we present the various notions of supercompact family of an implicative algebra introduced in \Cref{sec_topological_notions_IA} using the logical language underlying the notion of tripos. Since the properties considered in subsections \ref{SKsec} and \ref{Indsec} are stable under equivalence $\equiv_{\Sigma[I]}$ we will abuse of notation in the following results by considering predicates as functions rather than equivalence classes of functions.

We start by fixing the following notation:
\begin{definition}
     Let $\hyperdoctrine{\set}{\sP}$ be an implicative tripos. A predicate $\phi$ of $\sP(I\times J)$ is said to be 
    a \textbf{functional predicate } if
     \[P_{\angbr{\pr_1}{\pr_2}}(\phi)\wedge P_{\angbr{\pr_1}{\pr_3}}(\phi)\leq P_{\angbr{\pr_2}{\pr_3}}(\delta_J)\]
     where the domain of the projections is $I\times J\times J$;
     
\end{definition}

\begin{definition}\label{def_supercomp_in_triposes}
    Let $\hyperdoctrine{\set}{\sP}$ be an implicative tripos. A predicate $\varphi$ of $\sP(I)$ is:
    \begin{itemize}
        \item a \textbf{supercompact predicate} $\mathbf{(SK_p)}$ if whenever $\varphi \leq \exists_f (\psi)$ with $\freccia{J}{f}{I}$ and $\psi$ element of $\sP(J)$, there exists a function $\freccia{I}{g}{J}$ such that $\varphi \leq \sP_g (\psi)$ and $f\circ g=\id_I$;
        \item a \textbf{functionally supercompact predicate} $\mathbf{(fSK_p)}$ if for every functional predicate $\phi$ of $\sP (I\times  J)$, if $\varphi \leq \exists_{\pr_I}(\phi)$, then there exists a unique function $\freccia{I}{f}{J}$ such that $\varphi\leq \sP_{\angbr{\id_I}{f}}(\phi)$;
          \item a \textbf{weakly functionally supercompact predicate} $\mathbf{(wfSK_p)}$ if for every functional predicate $\phi$ of $\sP (I\times J)$, if $\varphi \leq \exists_{\pr_I}(\phi)$, then there exists a function $\freccia{I}{f}{J}$ such that $\varphi\leq \sP_{\angbr{\id_I}{f}}(\phi)$;
    \end{itemize}
\end{definition}
In the language of triposes, the ``uniformity'' property can be presented as a \emph{stability under re-indexing condition}:
\begin{definition}
    Let $\hyperdoctrine{\set}{\sP}$ be an implicative tripos. A predicate $\varphi$ of $\sP(I)$ is:
    \begin{itemize}
    
        \item  a \textbf{uniformly supercompact predicate}  $\mathbf{(U\mbox{-}SK_p)}$ if $\sP_f(\varphi) $ is a supercompact predicate for every function $\freccia{J}{f}{I}$;
        \item a \textbf{uniformly functionally supercompact predicate}  $\mathbf{(U\mbox{-}fSK_p)}$ if $\sP_f(\varphi) $ is a functionally supercompact predicate for every function $\freccia{J}{f}{I}$;
        \item a \textbf{uniformly weakly functionally supercompact predicate}  $\mathbf{(U\mbox{-}wfSK_p)}$ if $\sP_f(\varphi) $ is a weakly functionally supercompact predicate for every function $\freccia{J}{f}{I}$;
    \end{itemize}
\end{definition}
\begin{proposition}\label{pro_super_comp_in_tripos_vs_IA}
    Let $\hyperdoctrine{\set}{\sP}$ be an implicative tripos, and let $\varphi$ be a predicate of $\sP(I)$. We have that:
    \begin{enumerate}
        \item  $\varphi$ is $\mathbf{(SK_p)}$  if and only if $(\varphi (i))_{i\in I}$ is $\mathbf{(SK)}$;
        \item  $\varphi$ is  $\mathbf{(fSK_p)}$  if and only if $(\varphi (i))_{i\in I}$ is $\mathbf{(fSK)}$;
        \item  $\varphi$ is  $\mathbf{(wfSK_p)}$  if and only if $(\varphi (i))_{i\in I}$ is $\mathbf{(wfSK)}$.
        
    \end{enumerate}
\end{proposition}
\begin{proof}
The proofs are straightforward. We provide just the proof of the first point, since the other two follow by similar arguments.

 Suppose that $\varphi$ is $\mathbf{(SK_p)}$, and let us consider a family $((b_j^i)_{j\in J_i})_{i\in I}$, with
    \begin{equation}\label{eq_1_prop_ex_free_iff_I_unif_supercompact}
        \bigwedge_{i\in I}(\varphi(i)\to \bigexists_{j\in J_i} b^i_j)\in \Sigma.
    \end{equation}
    Now let us define by $g: \coprod_{i \in I} J_i\to I$ the function sending an element $(i,j)$ to $i$, and by $\phi\in \mathsf{P}(\coprod_{i \in I} J_i)$ the predicate sending $(i,j)$ to $ b^i_j$. By definition of the left adjoints $\exists_f$ in an implicative tripos, we have that \eqref{eq_1_prop_ex_free_iff_I_unif_supercompact} is equivalent to
      \begin{equation}\label{eq_2_prop_ex_free_iff_I_unif_supercompact}
       \varphi \vdash_{\Sigma [I]} \exists_g \phi.
    \end{equation}
    Hence, by definition of $\mathbf{(SK_p)}$, there exists a function $f:I\to \coprod_{i\in I}J_i$ such that $g\circ f=\id_I$, and 
    \begin{equation}\label{eq_3_prop_ex_free_iff_I_unif_supercompact}
       \varphi \vdash_{\Sigma [I]} \mathsf{P}_f (\phi).
    \end{equation}
    By definition, we have that \eqref{eq_3_prop_ex_free_iff_I_unif_supercompact} means 
     \begin{equation}\label{eq_4_prop_ex_free_iff_I_unif_supercompact}
   \bigwedge_{i\in I} ( \varphi (i) \to \phi(f(i)))\in \Sigma .
    \end{equation}
   By definition of $\phi$, and since the second component $f_2(i)$ of $f(i)$ is in $J_i$, we can conclude from \eqref{eq_4_prop_ex_free_iff_I_unif_supercompact} that
   \[ \bigwedge_{i\in I} ( \varphi (i) \to b^i_{f_2(i)})\in \Sigma \]
   i.e. that the family $(\varphi(i))_{i \in I}$ is $\mathbf{(SK)}$.
   
   Employing a similar argument one can check that the converse holds too.
\end{proof}
Using the previous result and the fact that $\mathsf{P}$ acts on arrows by reindexing we get the following proposition.
\begin{proposition}\label{FEF1}
    Let $\hyperdoctrine{\set}{\sP}$ be an implicative tripos, and let $\varphi$ be a predicate of $\sP(I)$. We have that:
    \begin{enumerate}
        \item  $\varphi$ is $\mathbf{(U\mbox{-}SK_p)}$  if and only if $(\varphi (i))_{i\in I}$ is $\mathbf{(U\mbox{-}SK)}$;
        \item  $\varphi$ is  $\mathbf{(U\mbox{-}fSK_p)}$  if and only if $(\varphi (i))_{i\in I}$ is $\mathbf{(U\mbox{-}fSK)}$;
        \item  $\varphi$ is  $\mathbf{(U\mbox{-}wfSK_p)}$  if and only if $(\varphi (i))_{i\in I}$ is $\mathbf{(U\mbox{-}wfSK)}$.
        
    \end{enumerate}
\end{proposition}
Combining the previous proposition with \Cref{cor_UwfSK_iff_UfSK} we obtain the following corollary:
\begin{corollary}\label{cor_UwfSK_p_iff_UfSK}
    A predicate of an implicative tripos is  $\mathbf{(U\mbox{-}wfSK_p)}$ if and only if it is  $\mathbf{(U\mbox{-}fSK_p)}$
\end{corollary}
\begin{remark}\label{rem_SK_p_are_FEF_in_realizability}
    Notice that the $\mathbf{(SK_p)}$ and $\mathbf{(U\mbox{-}SK_p)}$ predicates of a tripos are precisely the elements called \emph{full existential splitting} and \emph{full existential free} respectively in \cite{MaiettiTrotta21}. The $\mathbf{(U\mbox{-}SK_p)}$ predicates of a tripos coincide also with those predicates called $\exists$-prime predicates introduced in \cite{Frey2014AFS}. 
    
\end{remark}
\begin{example}
In the case of triposes for complete Heyting algebras with $\Sigma=\{\top\}$, we obtain that  $\mathbf{(SK_p)}$ and $\mathbf{(U\mbox{-}SK_p)}$ predicates coincide and are exactly predicates $\varphi$ such that $(\varphi (i))_{i\in I}$ is $\mathbf{(cSK)}$ since $\Sigma$ is clearly closed under arbitrary infima; hence we re-obtain \cite[Lem. 7.31]{MaiettiTrotta21}.
\end{example}

\begin{example} Combining \Cref{pro_super_comp_in_tripos_vs_IA} and  \Cref{FEF1} with examples in \ref{sect_supercompact_in_realizability}, 
we have that in the triposes arising from implicative algebras coming from realizability, relative realizability and nested realizability,  $\mathbf{(SK_p)}$ predicates coincide with $\mathbf{(U\mbox{-}SK_p)}$ predicates. In the case of realizability they are exactly those predicates (equivalent to) singleton predicates (that is predicates $\alpha$ for which $\alpha(x)$ is always a singleton). 
This characterization was already provided in \cite{MaiettiTrotta21}.





\end{example}





\section{Partitioned assemblies and assemblies for implicative triposes}
In this section we introduce the notion of \emph{partitioned assemblies} and \emph{assemblies} for implicative triposes. 

Before starting our analysis, we recap here some useful notions and results regarding the \emph{Grothendieck category} of an implicative tripos
\subsection{Some properties of the Grothendieck category of an implicative tripos}
\begin{definition}[Grothendieck category]
    Let $\hyperdoctrine{\set}{\sP}$ be an implicative tripos. The \textbf{Grothendieck category} $\Gamma[\mathsf{P}]$ of $\mathsf{P}$ is the category whose objects are pairs $(A,\alpha)$ with $\alpha\in \mathsf{P}(A)$ and whose arrows from $(A,\alpha)$ to $(B,\beta)$ are arrows $f:A\rightarrow B$ in $\mC$ such that $\alpha\leq \mathsf{P}_{f}(\beta)$.
\end{definition}
It is direct to check that for every implicative tripos $\hyperdoctrine{\set}{\sP}$, we have an adjunction 
\[\begin{tikzcd}
	\Gamma[\sP] && \set
	\arrow[""{name=0, anchor=center, inner sep=0}, "U", curve={height=-12pt}, from=1-1, to=1-3]
	\arrow[""{name=1, anchor=center, inner sep=0},"\Delta", curve={height=-12pt}, hook', from=1-3, to=1-1]
	\arrow["\dashv"{anchor=center, rotate=-90}, draw=none, from=0, to=1]
\end{tikzcd}\]
where $U(X,\varphi):=X$, $U(f):=f$, $\Delta (X):=(X,\top_X)$ and $\Delta(f):=f$.
\begin{proposition}\label{regepigamma}
 Let $\hyperdoctrine{\set}{\sP}$ be an implicative tripos. Then, the regular epis in $\Gamma[\mathsf{P}]$ are arrows $f:(A,\alpha)\rightarrow (B,\beta)$ such that $f$ is a regular epi in $\set$ and $\beta=\exists_f(\alpha)$.
\end{proposition}
\begin{proof}
Since the forgetful functor $U:\Gamma[\mathsf{P}]\rightarrow \set$ is left adjoint to the functor $\Delta:\set\rightarrow \Gamma[\mathsf{P}]$, $U$ preserves colimits. In particular, if $f:(A,\alpha)\rightarrow (B,\beta)$ is a regular epi, that is the coequalizer of two arrows $g,h:(C,\gamma)\rightarrow (A,\alpha)$, then $f:A\rightarrow B$ is a coequalizer of $g,h:C\rightarrow A$ in $\set$. Let us show now that $\beta=\exists_f(\alpha)$. Since $\alpha\leq \mathsf{P}_f(\beta)$, then $\exists_{f}(\alpha)\leq \beta$. Moreover, we know that $\alpha\leq \mathsf{P}_{f}(\exists_{f}(\alpha))$. Thus $f:(A,\alpha)\rightarrow (B,\exists_f(\alpha))$ is a well-defined arrow which coequalizes $g,h:(C,\gamma)\rightarrow (A,\alpha)$. This implies that $\id_A:(A,\beta)\rightarrow (A,\exists_f(\alpha))$ must be a well-defined arrow in $\Gamma[\sP]$. Thus, $\beta\leq \exists_f(\alpha)$. Hence we get $\beta=\exists_f(\alpha)$. 

Conversely, let $f:A\rightarrow B$ be a regular epi in $\set$. Then it is the coequalizer in $\set$ of two arrows $g,h:C\rightarrow A$. Since $\alpha\leq \sP_f\exists_f(\alpha)$, the arrow $f:(A,\alpha)\rightarrow(B,\exists_f(\alpha))$ is well-defined. It is immediate to verify that this arrow is a coequalizer for the arrows $g,h:(X,\sP_g(\alpha)\wedge \sP_h(\alpha))\rightarrow (A,\alpha)$.
\end{proof}
\begin{proposition}\label{gammafl} Let $\hyperdoctrine{\set}{\sP}$ be an implicative tripos. Then $\Gamma[\mathsf{P}]$ is regular.
\end{proposition}
\begin{proof}
We start by showing that $\Gamma[\mathsf{P}]$ has all finite limits. First, it is direct to check that  
$(1,\top_1)$ is a terminal object in $\Gamma[\mathsf{P}]$. A product of $(A,\alpha)$ and $(B,\beta)$ in $\Gamma[\mathsf{P}]$ is given by $(A\times B,\mathsf{P}_{\pi_1}(\alpha)\wedge\mathsf{P}_{\pi_2}(\beta) )$ together with the projections $\pi_1$ and $\pi_2$, while an equalizer of two parallel arrows $f,g:(A,\alpha)\rightarrow (B,\beta)$ in $\Gamma[\sP]$ is given by $(E,\mathsf{P}_e(\alpha))$ where $e:E\rightarrow A$ is an equalizer of $f,g$ in $\set$. 

Now we show that $\Gamma[\sP]$ is regular. Let $f:(A,\alpha)\rightarrow (B,\beta)$ be an arrow in $\Gamma[\sP]$. We can factorize $f:A\rightarrow B$ in $\set$ as a regular epi $r:A\rightarrow R$ followed by a mono $m:R\rightarrow B$. The arrow $r:(A,\alpha)\rightarrow (R,\exists_{r}(\alpha))$ is a regular epi in $\Gamma[\sP]$ by Proposition \ref{regepigamma} and $m:(R,\exists_{r}(\alpha))\rightarrow (B,\beta)$ is a mono. Such factorizations are unique up-to-isomorphism and pullback-stable in $\Gamma[\mathsf{P}]$. Thus $\Gamma[\sP]$ is a regular category.

\end{proof}
\begin{example}
    Let us consider an implicative tripos $\hyperdoctrine{\set}{\sP}$ for a complete Heyting algebra $\mathbb{H}$. Then the Grothendieck category $\Gamma [\sP]$ is the coproduct completion $\mathbb{H}_+$ of the category $\mathbb{H}$. This fact was observed in \cite{maiettitrotta2021arxiv}
\end{example}

\subsection{Partitioned assemblies}
The main purpose of this section is to generalize the notion of \emph{category of partitioned assemblies} associated with a PCA to implicative algebras and implicative triposes.

Let us recall that given a (partial) combinatory algebra $(R,\cdot)$ the category of \textbf{partitioned assemblies} is defined as follows:
\begin{itemize}
    \item an object is a pair $(X,\varphi)$ where $X$ is a set and $\freccia{X}{\varphi}{R}$ is a function from $X$ to the PCA;
    \item a morphism $\freccia{(X,\varphi)}{f}{(Y,\psi)}$ is a function $\freccia{X}{f}{Y}$ such that there exists an element $a\in R$ with $a\cdot \varphi(x)=\psi(  f(x))$ for every $x\in X$, i.e. there exists an element $a$ of the PCA such that the diagram
\[\begin{tikzcd}
	X & Y \\
	R & R
	\arrow["f", from=1-1, to=1-2]
	\arrow["\varphi"', from=1-1, to=2-1]
	\arrow["\psi", from=1-2, to=2-2]
	\arrow["a\cdot (-)"', dashed, from=2-1, to=2-2]
\end{tikzcd}\]
commutes.
\end{itemize}
It is proved in \cite{MaiettiTrotta21,maiettitrotta2021arxiv} that the category of partitioned assemblies can be completely defined in terms of realizability triposes and full existential free elements, namely it is the full subcategory of the Grothendieck category $\Gamma [\sP]$ whose second components is given by a full existential free element. 

Therefore we have that \Cref{rem_SK_p_are_FEF_in_realizability} suggests that the following definition provides a natural generalization of the ordinary notion of category of partitioned assemblies to an arbitrary implicative tripos:
\begin{definition}[partitioned assemblies]\label{def_partitioned_assemblies}
    Let $\hyperdoctrine{\set}{\sP}$ be an implicative tripos. We define the category of \textbf{partitioned assemblies}  $\mathbf{PAsm}_{\sP}$ of $\sP$ as the full sub-category of $\Gamma [\sP]$ given by the objects of $\Gamma [\sP]$ whose second component is a $(\mathbf{U\mbox{-}SK_p})$ predicate of $\sP$.
\end{definition}
Notice that in general the category of partitioned assemblies of an implicative tripos has no finite limits. This is due to the fact that, in general, $(\mathbf{U\mbox{-}SK_p})$ predicates are not closed under finite meets.
In fact, combining the stability under reindexing of $(\mathbf{U\mbox{-}SK_p})$ predicates with the definition of finite limits in $\Gamma [\sP]$ (see \Cref{gammafl}), it is straightforward to check that:
\begin{lemma}\label{lem_PAsm_lex_iff}
The category $\mathbf{PAsm}_{\sP}$ is a lex subcategory of $\Gamma[\sP]$ if and only if $(\mathbf{U\mbox{-}SK_p})$ predicates are closed under finite meets, that is if and only if $(a_i\times b_i)_{i\in I}$ is $\mathbf{U\mbox{-}SK}$ for every pair of $\mathbf{U\mbox{-}SK}$ families $(a_i)_{i\in I}$ and $(b_i)_{i\in I}$.
\end{lemma}
\begin{example}
    In the case of realizability triposes the category defined in \Cref{def_partitioned_assemblies} coincides with the ordinary category of partitioned assemblies.
\end{example}
\begin{example}
    Let $\hyperdoctrine{\set}{\sP}$ be an implicative tripos associated with a complete Heyting algebra. Then, we have that an object of $\mathbf{PAsm}_{\sP}$ is a pair $(X,\varphi)$ such that every element $\varphi (x)$ is supercompact in the sense of  \cite{BANASCHEWSKI199145}.
\end{example}
\begin{remark}
    A nice intrinsic characterization of categories which are equivalent to a category of partitioned assemblies for a PCA is presented in \cite[Thm. 3.8]{FreyRT}, where the author proves that  a category is equivalent to partitioned assemblies over a PCA if and only if it is w.l.c.c. and well-pointed local, and has a discrete generic object. This intrinsic description, combined with  the notion of category of partitioned assemblies for an implicative algebra, offers a useful tool to identify implicative algebras whose category of partitioned assemblies happens to be equivalent to the category of partitioned assemblies for a PCA. Again, these considerations can be extended to the case of categories of assemblies for implicative algebras, which we will define in the next section, taking advantage of the intrinsic description of the regular completion of a lex category \cite{SFEC}. 
\end{remark}
\subsection{Assemblies}
The main purpose of this section is to generalize the notion of \emph{category of assemblies} associated with a PCA to implicative algebras and implicative triposes.

Let us recall (see for example \cite{van_Oosten_realizability}) that given a partial combinatory algebra $(R,\cdot)$ the category of \textbf{assemblies} is defined as follows:
\begin{itemize}
    \item an object is a pair $(X,\varphi)$ where $X$ is a set and $\freccia{X}{\varphi}{\powerset^* (R)}$ is a function from $X$ to the non-empty powerset of the  PCA;
    \item a morphism $\freccia{(X,\varphi)}{f}{(Y,\psi)}$ is a function $\freccia{X}{f}{Y}$ such that there exists an element $a\in R$ with $a\cdot \varphi(x)\subseteq \psi(f(x))$ for every $x\in X$.
    \end{itemize}
Notice that, by the result presented in \Cref{sect_supercompact_in_realizability}, we have that the category of assemblies can be described as the full subcategory of $\Gamma [\sP]$ associated with the realizability tripos, whose objects are given by $(X,\varphi)$  where $\varphi$ enjoys the property of being $(\mathbf{U\mbox{-}fSK_p})$, or equivalently (by \Cref{cor_UwfSK_p_iff_UfSK}), of being $(\mathbf{U\mbox{-}wfSK_p})$.

This correspondence between assemblies and   $(\mathbf{U\mbox{-}fSK_p})$ predicates of realizability triposes suggests the following abstraction of the notion of assemblies:
\begin{definition}[Assemblies]\label{def_assemblies}
    Let $\hyperdoctrine{\set}{\sP}$ be an implicative tripos. We define the category of \textbf{assemblies}  $\mathbf{Asm}_{\sP}$ of $\sP$ as the full sub-category of $\Gamma [\sP]$ given by the objects of $\Gamma [\sP]$ whose second component is a $(\mathbf{U\mbox{-}fSK_p})$ predicate of $\sP$.
\end{definition}

Hence we have the following inclusions of categories:
\[\begin{tikzcd}	
\mathbf{PAsm}_{\mathsf{P}} & \mathbf{Asm}_{\mathsf{P}}& \Gamma[\sP]
\arrow[hook,from=1-1, to=1-2]
	\arrow[hook, from=1-2, to=1-3]
\end{tikzcd}\]

As in the case of partitioned assemblies, the category of assemblies of an implicative tripos is not lex or regular, since $(\mathbf{U\mbox{-}fSK_p})$ predicates are not closed under finite meets in general.

\begin{proposition}
    The category $\mathbf{Asm}_{\sP}$ is a regular subcategory of $\Gamma[\sP]$ if and only if:
    \begin{itemize}
        \item $(\mathbf{U\mbox{-}fSK_p})$ predicates are closed under finite meets;
         \item  $(\mathbf{U\mbox{-}fSK_p})$ predicates are stable under existential quantifiers along regular epis, i.e. for every $(\mathbf{U\mbox{-}fSK_p})$ predicate $\varphi$ and $r$ regular epi of $\set$ we have $\exists_r(\varphi )$ is $(\mathbf{U\mbox{-}fSK_p})$.
    \end{itemize}

\end{proposition}
\begin{proof}
    As in the case of partitioned assemblies, we have that $\mathbf{Asm}_{\sP}$ is a lex sub-category of  $\Gamma[\sP]$ if and only if  $(\mathbf{U\mbox{-}fSK_p})$ predicates are closed under finite meets. To conclude the proof, it is enough to observe that the factorization system of $\Gamma[\sP] $ induces a factorization system on $\mathbf{Asm}_{\sP}$  if and only if $(\mathbf{U\mbox{-}fSK_p})$ predicates are stable under existential quantifiers along regular epis. But this follows by the explicit description of the factorization system of  $\Gamma[\sP] $, i.e. we have that an arrow $\freccia{(A,\alpha)}{f}{(B,\beta)}$ can be written as
\[\begin{tikzcd}
	{(A,\alpha)} && {(B,\beta)} \\
	& {(R,\exists_r(\alpha))}
	\arrow["r"', two heads, from=1-1, to=2-2]
	\arrow["m"', tail, from=2-2, to=1-3]
	\arrow["f", from=1-1, to=1-3]
\end{tikzcd}\]
with $r$ regular epi and $m$ mono.
\end{proof}
\begin{example}
When $\hyperdoctrine{\set}{\sP}$ is a realizability tripos, the category $\mathbf{Asm}_{\sP}$  coincides with the ordinary category of assemblies for a CA, as described in \cite{van_Oosten_realizability}. 
\end{example}

\begin{example}
When $\hyperdoctrine{\set}{\sP}$ is an implicative tripos for nested realizability the category $\mathbf{Asm}_{\sP}$  coincides with the category of assemblies for nested realizability, as described in \cite[Sec. 1.3]{maschiostreicher15}.
\end{example}
\begin{example}
    When $\hyperdoctrine{\set}{\sP}$ is an implicative tripos for a complete Heyting algebra we have that the category $\mathbf{Asm}_{\sP}$ is given by objects $(I,\varphi)$ where, for every $i\in I$, $\varphi (i)$ is indecomposable, 
    see  \Cref{sect_supercompact_in_realizability}.
\end{example}
\begin{example}
    When $\hyperdoctrine{\set}{\sP}$ is an implicative tripos for a complete Boolean algebra we have that $\mathbf{Asm}_{\sP}\cong \mathbf{PAsm}_{\sP}$, since in every complete Boolean algebra we have that $(\mathbf{U\mbox{-}fSK_p})\equiv (\mathbf{U\mbox{-}SK_p})$, see  \Cref{sect_supercompact_in_heyting_and_boolean_alg}.
\end{example}

\subsection{Regular completion of implicative triposes}
It is a well-known result (see \cite{pitts1981theory,UEC,TECH,van_Oosten_realizability}) that every topos $\mC[\mathsf{P}]$ obtained as the result of the tripos-to-topos construction from a given tripos $\mathsf{P}$ can be presented as the $\mathsf{ex}/\mathsf{reg}$-completion (according to \cite{REC})
$$\mC[\mathsf{P}]\simeq \left(\mathbf{Reg}_{\mathsf{P}}\right)_{\mathsf{ex}/\mathsf{reg}}$$
of a certain regular category, that we denote by $\mathbf{Reg}_{\mathsf{P}}$, constructed from the tripos $\mathsf{P}$. By the universal property of the $(-)_{\mathsf{ex}/\mathsf{reg}}$ completion the canonical embedding $$\freccia{\mathbf{Reg}_{\mathsf{P}}}{\mathbf{y}}{\mC[\mathsf{P}]}$$ is a full and faithful regular functor.

The universal properties of the category $\mathbf{Reg}_{\mathsf{P}}$ (which is called $\mathsf{Ass}_\mathcal{C}(\mathsf{P})$ in \cite{van_Oosten_realizability}) are analysed in detail in \cite{UEC,TECH},  where it is proved that such a category enjoys the property of being the \emph{regular completion} of $\mathsf{P}$.

In the following definition, we recall an explicit description of such a category (in the case of $\set$-based triposes):
%

\begin{definition}\label{def: Ass P}
	Let $\mathsf{P}:\set^{op}\rightarrow \mathsf{Hey}$ be an implicative tripos. We define the category $\mathbf{Reg}_{\mathsf{P}}$ as follows:
\begin{itemize}
	\item the \textbf{objects} of $\mathbf{Reg}_{\mathsf{P}}$ are pairs $(A,\alpha)$, where $A$ is a set and $\alpha$ is an element of $P(A)$;
	\item an \textbf{arrow} of $\mathbf{Reg}_{\mathsf{P}}$ from $(A,\alpha)$ to $(B,\beta)$ is given by an element $\phi$ of $P(A\times B)$ such that:
	\begin{enumerate}
		\item $\phi \leq P_{\pr_1}(\alpha)\wedge P_{\pr_2}(\beta)$;
		\item $\alpha\leq \exists_{\pr_1}(\phi)$;
		\item $P_{\angbr{\pr_1}{\pr_2}}(\phi)\wedge P_{\angbr{\pr_1}{\pr_3}}(\phi)\leq P_{\angbr{\pr_2}{\pr_3}}(\delta_B)$.
	\end{enumerate}
	\end{itemize}
	The compositions of morphisms of $\mathbf{Reg}_{\mathsf{P}}$ is given by the usual \emph{relational composition}: the composition of $\freccia{(A,\alpha)}{\phi}{(B,\beta)}$ and $\freccia{(B,\beta)}{\psi}{(C,\gamma)}$ is given by 
	\[ \exists_{\angbr{\pr_1}{\pr_3}}(P_{\angbr{\pr_1}{\pr_2}}(\phi)\wedge  P_{\angbr{\pr_2}{\pr_3}}(\psi))\]
	where $\pr_i$ for $i=1,2,3$ are projections from $A\times B\times C$.
\end{definition}
\begin{remark}\label{rem:morphisms of assemblies}
    Notice that, from 1.\ and 2.\ in the definition, every arrow $\phi:(A,\alpha)\to(B,\beta) $ of $\mathbf{Reg}_{\mathsf{P}}$ satisfies $\alpha=\exists_{\pr_1}(\phi)$.
\end{remark}
\begin{remark}\label{rem_morphism_of_Gr_induces_morph_ass}
Notice that one can always define a finite-limit preserving functor as follows
$$\mathbf{F}:\Gamma[\mathsf{P}]\rightarrow \mathbf{Reg}_{\mathsf{P}}$$
$$(A,\alpha)\mapsto (A,\alpha)$$
$$f\mapsto\exists_{\langle \id_A,f\rangle}(\alpha)$$
for every arrow $f:(A,\alpha)\rightarrow (B,\beta)$. It is  straightforward to check that this functor is well-defined. Indeed, for every arrow $f:(A,\alpha)\rightarrow (B,\beta)$ in $\Gamma[\mathsf{P}]$,  $\exists_{\langle \id_A,f\rangle}(\alpha)$ satisfies condition $1.$ since $\alpha\leq \mathsf{P}_{f}(\beta)$ and satisfies $2.$ by its very definition; moreover it satisfies also condition $3.$ as one can easily see by using adequately adjunctions and Heyting implications, and exploiting BCC and FR. Identities $\mathsf{Id}_{(A,\alpha)}$ are sent to $\exists_{\Delta_A}(\alpha)$, that is to identities in $\mathbf{Reg}_{\mathsf{P}}$. Finally, the composition is preserved as one can prove using FR and BCC.

The functor $\mathbf{F}$ is not faithful.  Indeed, for every pair of sets $A,B$, we have that $\mathsf{Hom}_{\Gamma[\mathsf{P}]}((A,\bot),(B,\top))=\mathsf{Hom}_{\set}(A,B)$, while $ \mathsf{Hom}_{\mathbf{Reg}_{\mathsf{P}}}((A,\bot),(B,\top))\simeq \{\star\}$. If $A$ is non-empty and $B$ has at least two elements, we have that $$\mathsf{Hom}_{\Gamma[\mathsf{P}]}((A,\bot),(B,\top))\not\simeq\mathsf{Hom}_{\mathbf{Reg}_{\mathsf{P}}}((A,\bot),(B,\top))$$

In general, $\mathbf{F}$ is neither provable to be full. E.g.\ consider the case of the tripos induced by a Boolean algebra with four elements $\{\bot, a, \neg a, \top\}$ and the arrow $\phi:(\{0\}, \top)\rightarrow (\{0,1\},\top)$ defined by 
$$\phi(0,x)=\begin{cases}
a\textrm{ if }x=0\\
\neg a\textrm{ if }x=1\\
\end{cases}$$
This is a well-defined arrow in $\mathbf{Reg}_{\mathsf{P}}$ which however has not the form $\mathbf{F}(f)$ for any $f:\{0\}\rightarrow \{0,1\}$.


Notice that in general the category $\mathbf{Reg
}_{\mathsf{P}}$  is not (equivalent to) a full subcategory of the Grothenieck category $\Gamma[\mathsf{P}]$, since morphisms of $\mathbf{Reg}_{\mathsf{P}}$ may not arise from morphisms of the base category. 
\end{remark}
\begin{lemma}\label{regepiasm}
    A morphism  $\freccia{(A,\alpha)}{\phi}{(B,\beta)}$ in $\mathbf{Reg}_{\mathsf{P}}$ is a regular epi if and only if $\beta=\exists_{\pr_2}(\phi)$.
\end{lemma}
\begin{proof}
We know that the embedding of $\mathbf{Reg}_{\mathsf{P}}$ into $\set[\mathsf{P}]\simeq (\mathbf{Reg}_{\mathsf{P}})_{\mathsf{ex}/\mathsf{reg}}$ preserves regular epis, since it is a regular functor. Moreover, since the embedding preserves finite limits and every regular epi in $\set[\mathsf{P}]$ is a coequalizer of its kernel pair, we can conclude that the embedding also reflects regular epis. Thus an arrow in $\mathbf{Reg}_{\mathsf{P}}$ is a regular epi if and only if it is a regular epi in $\set[\mathsf{P}]$.

We know that in $\set[\mathsf{P}]$ every epi is regular since it is a topos. Thus using the characterization of epis of $\set[\mathsf{P}]$ in \cite{pitts1981theory}, we can conclude.

\end{proof}

\begin{remark} \label{prop_F_is_regular}
Notice that the functor $\mathbf{F}:\Gamma[\sP]\rightarrow \mathbf{Reg}_{\mathsf{P}}$ preserves regular epis. In particular, 
the functor $\mathbf{F}$ is regular. This follows immediately from Proposition \ref{regepigamma}, Lemma \ref{regepiasm} and Proposition \ref{gammafl}.

\end{remark}
\begin{remark}
Notice that the category $\mathbf{Reg}_{\mathsf{P}}$ of a tripos $\mathsf{P}$ can also be described by means of the constant object functor as the full subcategory of $\set[\mathsf{P}]$ of subobjects of constant objects $\Delta(A)$ for some object $A$ of $\set$ (for the definition of the constant object functor $\Delta:\set\rightarrow \set[\sP]$ see e.g.\ \cite{van_Oosten_realizability}).
\end{remark}
\begin{remark}\label{lessimpliesequal}
Notice that in $\mathbf{Reg}_{\mathsf{P}}$, as observed in \cite{pitts1981theory}, for parallel arrows $\phi,\psi:(A,\alpha)\to (B,\beta)$, we have that $\phi\leq \psi$ if and only if $\phi=\psi$. 
Here we sketch a proof that if $\phi\leq \psi$ then $\psi\leq\phi$. By \Cref{rem:morphisms of assemblies} we have that $\exists_{\pr_1}(\phi)=\alpha=\exists_{\pr_1}(\psi)$ and then, in particular, we have that $\exists_{\pr_1}(\psi)\leq \exists_{\pr_1}(\phi)$, and then that $\psi\leq \sP_{\pr_1}\exists_{\pr_1}(\phi)$. By BCC we have that $\psi\leq \exists_{\angbr{\pi_1}{\pi_2}} \sP_{\angbr{\pi_1}{\pi_3}}(\phi)$, where $\pi_i$ are the projections in the following pullback
\[\begin{tikzcd}
	{A\times B\times B} & {A\times B} \\
	{A\times B} & A
	\arrow["{\pi_1}", from=2-1, to=2-2]
	\arrow["{\pi_1}", from=1-2, to=2-2]
	\arrow["{\angbr{\pi_1}{\pi_3}}"', from=1-1, to=2-1]
	\arrow["{\angbr{\pi_1}{\pi_2}}", from=1-1, to=1-2]
\end{tikzcd}\]
In particular, we have that $\psi=\exists_{\angbr{\pi_1}{\pi_2}} \sP_{\angbr{\pi_1}{\pi_3}}(\phi) \wedge \psi$ and, by FR, we have that
 $$\psi=\exists_{\angbr{\pi_1}{\pi_2}} (\sP_{\angbr{\pi_1}{\pi_3}}(\phi) \wedge \sP_{\angbr{\pi_1}{\pi_2}}(\psi)).$$
 But since $\phi\leq \psi$, we have that
 \[\exists_{\angbr{\pi_1}{\pi_2}} (\sP_{\angbr{\pi_1}{\pi_3}}(\phi) \wedge \sP_{\angbr{\pi_1}{\pi_2}}(\psi))=\exists_{\angbr{\pi_1}{\pi_2}} (\sP_{\angbr{\pi_1}{\pi_3}}(\phi) \wedge \sP_{\angbr{\pi_1}{\pi_3}}(\psi) \wedge \sP_{\angbr{\pi_1}{\pi_2}}(\psi)).\]
 Since $\psi$ is functional, i.e. $\sP_{\angbr{\pi_1}{\pi_3}}(\psi) \wedge \sP_{\angbr{\pi_1}{\pi_2}}(\psi)\leq \sP_{\angbr{\pr_2}{\pr_3}}(\delta_B)$, we have that 
 $$\psi\leq \exists_{\angbr{\pi_1}{\pi_2}} (\sP_{\angbr{\pi_1}{\pi_3}}(\phi) \wedge \sP_{\angbr{\pr_2}{\pr_3}}(\delta_B)).$$
Employing the fact that $\delta_{B}=\exists_{\Delta_B}(\top_B)$, BCC and FR,  it is straightforward to check that
 \[\phi=\exists_{\angbr{\pi_1}{\pi_2}} (\sP_{\angbr{\pi_1}{\pi_3}}(\phi) \wedge \sP_{\angbr{\pr_2}{\pr_3}}(\delta_B)).\]
 Therefore we can conclude that $\psi\leq \phi$, and hence that $\psi = \phi$ (since $\phi\leq \psi$ by hypothesis).
\end{remark}


\subsubsection{The subcategory of trackable objects}
Let $\mathsf{P}:\set\rightarrow \mathsf{Hey}$ be a fixed implicative tripos for the rest of this section.

\begin{definition}[trackable morphism]
    Let $\freccia{(A,\alpha)}{\phi}{(B,\beta)}$ be a morphism of $\mathbf{Reg}_{\mathsf{P}}$. We say that $\phi$ is \textbf{trackable} if there exists a morphism $\freccia{A}{f_\phi}{B}$ of the base category such that $\alpha\leq\mathsf{P}_{\angbr{\id_A}{f_{\phi}}}(\phi)$.
\end{definition}
\begin{definition}[trackable object]
    An object $(A,\alpha)$ of $\mathbf{Reg}_{\mathsf{P}}$ is said to be a \textbf{trackable object} if every morphism  $\freccia{(A,\alpha)}{\phi}{(B,\beta)}$  of $\mathbf{Reg}_{\mathsf{P}}$ is trackable. We denote by  $\mathbf{Track}_{\mathsf{P}}$  the full subcategory of  $\mathbf{Reg}_{\mathsf{P}}$  whose objects are trackable assemblies.
\end{definition}
\begin{remark}\label{rem_track_morph_induces_morph_gr}
Notice that for $\freccia{(A,\alpha)}{\phi}{(B,\beta)}$ in $\mathbf{Reg}_{\mathsf{P}}$, if $\alpha\leq\mathsf{P}_{\angbr{\id_A}{f_{\phi}}}(\phi)$, then  $\alpha=\mathsf{P}_{\angbr{\id_A}{f_{\phi}}}(\phi)$, since the opposite inequality follows from $\phi\leq \mathsf{P}_{\pi_1}(\alpha)$.
    Notice moreover that when a morphism $\freccia{(A,\alpha)}{\phi}{(B,\beta)}$ of $\mathbf{Reg}_{\mathsf{P}}$ is trackable, then we have that the arrow $\freccia{A}{f_{\phi}}{B}$ induces a well-defined arrow $\freccia{(A,\alpha)}{f_{\phi}}{(B,\beta)}$ in $\Gamma[\mathsf{P}]$. In fact, by definition of arrows in $\mathbf{Reg}_{\mathsf{P}}$ we have that $\phi\leq \mathsf{P}_{\pr_1}(\alpha)\wedge \mathsf{P}_{\pr_2}(\beta)$, and then, by applying $\mathsf{P}_{\angbr{\id_A}{f_{\phi}}}$, we have that
    \[\alpha\leq\mathsf{P}_{\angbr{\id_A}{f_{\phi}}}(\phi)\leq \alpha \wedge \mathsf{P}_{f_{\phi}}(\beta)\]
    and then we can conclude that $\alpha\leq \mathsf{P}_{f_{\phi}}(\beta)$.  We can also notice that $\mathbf{F}(f_\phi)=\exists_{\langle \delta_A,f_\phi\rangle}(\alpha)\leq \phi$, by adjunction. Since both $\mathbf{F}(f_\phi)$ and $\phi$ are arrows in $\mathbf{Reg}_{\mathsf{P}}$ from $(A,\alpha)$ to $(B,\beta)$, we conclude by Remark \ref{lessimpliesequal} that they are in fact equal.  
\end{remark}
Now we can employ the notions introduced in \Cref{def_supercomp_in_triposes} to easily characterize the category of trackable objects of an implicative tripos:
\begin{proposition}
    Let $\hyperdoctrine{\set}{\sP}$ be an implicative tripos. Then an object $(A,\alpha)$ of $\mathbf{Reg}_{\sP}$ is trackable if and only if $\alpha$ is a $\mathbf{(wfSK_p)}$ predicate of $\sP$.
\end{proposition}
\subsubsection{The subcategory of strongly trackable objects}

\begin{definition}[strongly trackable morphism]
    Let $\freccia{(A,\alpha)}{\phi}{(B,\beta)}$ be a morphism of $\mathbf{Reg}_{\mathsf{P}}$. We say that $\phi$ is \textbf{strongly trackable} if there exists a unique morphism $\freccia{A}{f_{\phi}}{B}$ of the base category such that $\alpha\leq\mathsf{P}_{\angbr{\id_A}{f_{\phi}}}(\phi)$.
\end{definition}

\begin{definition}
    An object $(A,\alpha)$ of $\mathbf{Reg}_{\mathsf{P}}$ is said to be a \textbf{strongly trackable object} if every morphism  $\freccia{(A,\alpha)}{\phi}{(B,\beta)}$  of $\mathbf{Reg}_{\mathsf{P}}$ is strongly trackable. We denote by  $\mathbf{STrack}_{\mathsf{P}}$  the full subcategory of  $\mathbf{Reg}_{\mathsf{P}}$  strongly trackable objects.
\end{definition}


Employing the notions introduced in \Cref{def_supercomp_in_triposes} we can  easily characterize the category of strongly trackable objects of an implicative tripos:
\begin{proposition}\label{Prop_strong_tr_iff_fSK}
    Let $\hyperdoctrine{\set}{\sP}$ be an implicative tripos. Then an object $(A,\alpha)$ of $\mathbf{Reg}_{\sP}$ is strongly trackable if and only if $\alpha$ is a $\mathbf{(fSK_p)}$ predicate of $\sP$.
\end{proposition}
Hence we have the following diagram for every implicative tripos:
\[\begin{tikzcd}	
\mathbf{PAsm}_{\mathsf{P}} & \mathbf{Asm}_{\mathsf{P}}& \Gamma[\sP]\\
& \mathbf{STrack}_{\mathsf{P}} & \mathbf{Track}_{\mathsf{P}} & \mathbf{Reg}_{\mathsf{P}}
\arrow[hook,from=1-1, to=2-2]
	\arrow[hook, from=2-2, to=2-3]
 	\arrow[hook,from=2-3, to=2-4]
  	\arrow[hook,from=1-2, to=2-2]
 \arrow[hook, from=1-1, to=1-2]
 \arrow[ from=1-3, to=2-4, "\mathbf{F}"]
 \arrow[hook, from=2-2, to=1-3]
\end{tikzcd}\]

\subsection{Category of (regular) projective strongly trackable objects}
In the previous sections we have seen that the notions of $\mathbf{(fSK_p)}$ and $\mathbf{(wfSK_p)}$ predicates have a clear interpretation in terms of  trackable objects of the regular completion of an implicative tripos. The main purpose of this section is to show that the notion of $\mathbf{(SK_p)}$ predicates corresponds exactly to those strongly trackable objects of $\mathbf{Reg}_{\sP}$ that are \emph{regular projectives}.
\begin{definition}
    We denote by $\mathbf{Pr\mbox{-}STrack}_{\sP}$ the full subcategory of $\mathbf{Reg}_{\sP}$ whose objects are strongly trackable and regular projective.
\end{definition}

\begin{proposition}
    Every object $(A,\alpha)$ where $\alpha$ is $\mathbf{(SK_p)}$ is regular projective in $\mathbf{Reg}_{\mathsf{P}}$.
\end{proposition}
\begin{proof}
    Let us consider the following diagram
\[\begin{tikzcd}
	&& {(C,\gamma)} \\
	\\
	{(A,\alpha)} && {(B,\beta)}
	\arrow["{\phi_1}"', from=3-1, to=3-3]
	\arrow["{\phi_2}", from=1-3, to=3-3]
 	\arrow["{\phi_3}", dashed, from=3-1, to=1-3]
\end{tikzcd}\]
with $\phi_2$ regular epi in $\mathbf{Reg}_{\mathsf{P}}$, i.e. $\beta\leq \exists_{\pr_2}(\phi_2)$, and $\alpha$ $\mathbf{(SK_p)}$. We have to show that there exists a morphism $\phi_3$ such that the previous diagram commutes. By \Cref{{Prop_strong_tr_iff_fSK}}, we know that $\phi_1$ is trackable, i.e. there exists an arrow $\freccia{A}{f_{\phi_1}}{B}$ such that $\alpha=\sP_{\angbr{\id_A}{f_{\phi_1 }}}(\phi_1)$. By \Cref{rem_track_morph_induces_morph_gr}, we have that $\phi_1$ trackable implies that $\alpha\leq \sP_{f_{\phi_1}}(\beta)$. 
Thus, we have that
\begin{equation}\label{eq:crucial}
    \alpha\leq \sP_{f_{\phi_1}}(\beta)\leq  \sP_{f_{\phi_1}}\exists_{\pr_2}(\phi_2).
\end{equation}
By Beck-Chevalley condition, we have that \eqref{eq:crucial} implies that
\begin{equation}\label{eq:crucial 1}
    \alpha\leq  \exists_{\pr_2} (\sP_{\id_C\times f_{\phi_1} }(\phi_2)).
\end{equation}
Since $\alpha$ is $\mathbf{(SK_p)}$ we have that there exists an arrow $\freccia{A}{h}{C}$ such that

\begin{equation}\label{eq_crucial_3}
\alpha\leq \sP_{\angbr{h}{\id_A}}( \sP_{\id_C\times f_{\phi_1}}(\phi_2))=\sP_{\angbr{h}{f_{\phi_1}}}(\phi_2)
\end{equation}
Now we claim that $\phi_3:=\exists_{\angbr{\id_A}{h}}(\alpha)$ is a morphism of $\mathbf{Reg}_{\sP}$. Notice that 
it is enough to prove that $\alpha\leq \sP_h(\gamma)$ because if $\freccia{(A,\alpha)}{h}{(C,\gamma)}$ is a morphism in the Grothendieck category $\Gamma [\sP]$ then $\exists_{\angbr{\id_A}{h}}(\alpha)$ is a morphism of $\mathbf{Reg}_{\sP}$ from $(A,\alpha)$ to $(C,\gamma)$. 

Recall that since $\phi_2$ is a morphism of $\mathbf{Reg}_{\sP}$  we have, in particular, that $\phi_2\leq \sP_{\pr_1}(\gamma)$, and then we can combine this with \eqref{eq_crucial_3} to conclude that
\[\alpha\leq \sP_{\angbr{h}{f_{\phi_1}}}(\phi_2)\leq \sP_{\angbr{h}{f_{\phi_1}}}(\sP_{\pr_1}(\gamma))=\sP_{h}(\gamma).\]
Finally, let us check that the starting diagram commutes with $\phi_3$.

Recall that the composition $\phi_2\circ \phi_3$ in $\mathbf{Reg}_{\mathsf{P}}$ is given by
\begin{equation}\label{eq composition phi 2 and phi 3}
    \exists_{\angbr{\pr_1}{\pr_3}}(\sP_{\angbr{\pr_1}{\pr_2}}(\phi_3)\wedge\sP_{\angbr{\pr_2}{\pr_3}}(\phi_2))
\end{equation}
By definition of $\phi_3=\exists_{\angbr{\id_A}{h}}(\alpha)$, so we have that 
 \[\sP_{\angbr{\pr_1}{\pr_2}}(\phi_3)\wedge\sP_{\angbr{\pr_2}{\pr_3}}(\phi_2)= \sP_{\angbr{\pr_1}{\pr_2}}\exists_{\angbr{\id_A}{h}}(\alpha)\wedge\sP_{\angbr{\pr_2}{\pr_3}}(\phi_2) \]
 and by BCC, this is equal to
 \[ \exists_{\angbr{\pr_1}{h\circ\pr_1,\pr_2}}\sP_{\pr_1}(\alpha)\wedge\sP_{\angbr{\pr_2}{\pr_3}}(\phi_2) \]
 Now we can apply FR, obtaining 
  \[ \exists_{\angbr{\pr_1}{h\circ\pr_1,\pr_2}}(\sP_{\pr_1}(\alpha)\wedge\sP_{\angbr{h\circ\pr_1}{\pr_2}}(\phi_2) )\]
Therefore we have that  \eqref{eq composition phi 2 and phi 3} is equal to
  \[\sP_{\pr_1}(\alpha)\wedge\sP_{\angbr{h\circ\pr_1}{\pr_2}}(\phi_2)\]
By \Cref{lessimpliesequal},  to show that $\phi_3\circ\phi_2=\phi_1$ it is enough to show that $\phi_1\leq \phi_3\circ \phi_2$, i.e. that
\begin{equation}\label{eq phi1 leq phi3phi2}
     \phi_1\leq \sP_{\pr_1}(\alpha)\wedge\sP_{\angbr{h\circ\pr_1}{\pr_2}}(\phi_2).
\end{equation}
First, $\phi_1\leq \sP_{\pr_1}(\alpha)$ since $\phi_1$ is an arrow with domain $(A,\alpha)$ in $\mathbf{Reg}_{\mathsf{P}}$.

Now we show that $ \phi_1\leq \sP_{\angbr{h\circ\pr_1}{\pr_2}}(\phi_2)$: using again the fact that $\phi_1=\exists_{\angbr{\id_A}{f_{\phi_1}}}(\alpha)$ we have that 
\[\phi_1\leq \sP_{\angbr{h\circ\pr_1}{\pr_2}}(\phi_2)\iff \exists_{\angbr{h\circ\pr_1}{\pr_2}}(\phi_1)\leq \phi_2\iff  \exists_{\angbr{h}{f_{\phi_1}}}(\alpha) \leq \phi_2\]
and then we can conclude that
\[\phi_1\leq \sP_{\angbr{h\circ\pr_1}{\pr_2}}(\phi_2)\iff \alpha\leq \sP_{\angbr{h}{f_{\phi_1}}} (\phi_2).\]
Since we have that $ \varphi\leq \sP_{\angbr{h}{f_{\phi_1}}} (\phi_2)$ holds by \eqref{eq_crucial_3}, we can conclude that $\phi_1\leq \sP_{\angbr{h\circ\pr_A}{\pr_B}}(\phi_2)$. This concludes the proof that \eqref{eq phi1 leq phi3phi2} holds and then, by \Cref{lessimpliesequal}, that $\phi_1=\phi_3\circ\phi_2$.

\end{proof}

\begin{proposition}
  Let $(A,\alpha)$ be a strongly trackable object  of  $\mathbf{Reg}_{\sP}$. If $(A,\alpha)$ is regular projective, then $\alpha$ is $\mathbf{(SK_p)}$.
\end{proposition}
\begin{proof}
    Let us suppose that $\alpha\leq \exists_f (\beta)$ where $\freccia{B}{f}{A}$ is an arrow of the base category and $\beta \in \sP (B)$. Then, since $(A,\alpha)$ is regular projective, there exists an arrow $\phi$ such that the diagram
\[\begin{tikzcd}
	&& {(B,\beta)} \\
	\\
	{(A,\alpha)} && {(A,\exists_f(\beta))}
	\arrow["{\exists_{\Delta_A}(\alpha)}"', from=3-1, to=3-3]
	\arrow["{\exists_{\angbr{\id_B}{f}}(\beta)}", from=1-3, to=3-3]
	\arrow["\phi", dashed, from=3-1, to=1-3]
\end{tikzcd}\]
commutes in $\mathbf{Reg}_{\mathbf{P}}$ (indeed notice that $\exists_{\angbr{\id_B}{f}}(\beta)$ is a regular epi because $\exists_{\pr_2}\exists_{\angbr{\id_B}{f}}(\beta)=\exists_f(\beta)$). Since  $(A,\alpha)$ is a strongly trackable object there exists a unique $f_{\phi}:A\to B$ such that $ \phi=\exists_{\angbr{\id_A}{f_{\phi}}}(\alpha)$. Notice that since $\freccia{(A,\alpha)}{\phi}{(B,\beta)}$ is a morphism of assemblies we have that 
\[\phi=\exists_{\angbr{\id_A}{f_{\phi}}}(\alpha) \leq \sP_{\pr_1}(\alpha)\wedge \sP_{\pr_2}(\beta)\]
hence
\[\exists_{\angbr{\id_A}{f_{\phi}}}(\alpha) \leq \sP_{\pr_2}(\beta)\]
and then we can conclude that 
\[ \alpha \leq \sP_{f_{\phi}}(\beta).\]
Finally notice that the composition $f\circ f_{\phi}:A\to A$ has to be equal to the identity $\id_A$ on $A$ because ${\exists_{\angbr{\id_B}{f}}(\beta)}\circ \exists_{\angbr{\id_A}{f_{\phi}}}(\alpha)  =\exists_{\Delta_A}(\alpha)$ since the previous diagram commutes. In fact, first notice that  
\[{\exists_{\angbr{\id_B}{f}}(\beta)}\circ \exists_{\angbr{\id_A}{f_{\phi}}}(\alpha) =\exists_{\angbr{\id_A}{f\circ f_{\phi}}}(\alpha)\]
because by the definition of the functor $\mathbf{F}:\Gamma[\mathsf{P}]\rightarrow \mathbf{Reg}_{\mathsf{P}}$ we have that:
\[ \exists_{\angbr{\id_B}{f}}(\beta)\circ \exists_{\angbr{\id_A}{f_{\phi}}}(\alpha)=\mathbf{F}(f)\circ \mathbf{F}(f_{\phi})=\mathbf{F}(f\circ f_{\phi})=\exists_{\angbr{\id_A}{f\circ f_{\phi}}}(\alpha)\]
Now, since we have proved that $\exists_{\angbr{\id_A}{f\circ f_{\phi}}}(\alpha)={\exists_{\angbr{\id_B}{f}}(\beta)}\circ \exists_{\angbr{\id_A}{f_{\phi}}}(\alpha)  =\exists_{\Delta_A}(\alpha)$
we can use the uniqueness in the definition of strongly trackable morphism to conclude that $f\circ f_{\phi}=\id_A$. This allows us to conclude that $\alpha$ is $\mathbf{(SK_p)}$.
\end{proof}
As a corollary of the previous two propositions, we have that:
\begin{corollary}\label{cor_strong_track_regular_proj_iff}
A strongly trackable object $(A,\alpha)$ is regular projective in $\mathbf{Reg}_{\mathsf{P}}$ if and only if $\alpha$ is $\mathbf{(SK_p)}$. 

\end{corollary}

Summarizing, we have the following diagram:

\[\begin{tikzcd}
\mathbf{PAsm}_{\sP} & \mathbf{Asm}_{\sP} &\Gamma[\sP]\\
	\mathbf{Pr\mbox{-}STrack}_{\sP} & \mathbf{STrack}_{\sP} & \mathbf{Track}_{\sP} & \mathbf{Reg}_{\sP}
	\arrow[hook, from=2-1, to=2-2]
	\arrow[hook, from=2-2, to=2-3]
	\arrow[hook, from=2-3, to=2-4]
 \arrow[hook, from=1-1, to=1-2]
 \arrow[hook, from=1-1, to=2-1]
  \arrow[hook, from=1-2, to=2-2]
      \arrow[hook, from=2-2, to=1-3]
   \arrow[from=1-3, to=2-4, "\mathbf{F}"]
\end{tikzcd}\]

\subsection{A characterization of the categories of assemblies and regular completion}
It is well-known that in realizability the category of assemblies happens to be equivalent to the $\mathsf{reg}/\mathsf{lex}$-completion of its full subcategory of partition assemblies~\cite{SFEC}. In this section, we investigate for which implicative algebras we can extend this equivalence.
\begin{remark}
    Notice that, when we consider a tripos associated with a uniformly supercoherent implicative algebra, we have that, by \Cref{rem_SK_iff_U_SK_in_unif_suprc_alg}, $$\mathbf{PAsm}_{\sP}\equiv \mathbf{Pr\mbox{-}STrack}_{\sP}$$
\end{remark}
\begin{theorem}\label{thm_main_1}
    Let $\hyperdoctrine{\set}{\sP}$ be an implicative tripos, for a given implicative algebra $\mathbb{A}$. Then $\mathbb{A}$ is uniformly supercoherent if and only if $\mathbf{PAsm}_{\sP}$ is a lex (full) sub-category of $\mathbf{Reg}_{\sP}$ and it provides a projective cover of $\mathbf{Reg}_{\sP}$.
\end{theorem}
\begin{proof}
    Let us suppose that $\mathbb{A}$ is uniformly supercoherent (see \Cref{def_unif_supercoherent}). Then, since $\mathbf{U\mbox{-}SK}$ are closed under finite infs, we have that  $\mathbf{PAsm}_{\sP}$  is a lex subcategory of $\mathbf{Reg}_{\sP}$; this is a consequence of 
    \Cref{lem_PAsm_lex_iff} and its proof, of the proof of \Cref{gammafl} and of the fact that partitioned assemblies are strongly trackable.
    Now we show that  $\mathbf{PAsm}_{\sP}$ is a projective cover. We already know that the objects of  $\mathbf{PAsm}_{\sP}$ are projective (by \Cref{cor_strong_track_regular_proj_iff}), so we only have to show that every object $(B,\psi)$ of $\mathbf{Reg}_{\sP}$ of is covered by a regular projective of $\mathbf{PAsm}_{\sP}$. To show this we use the fact that in a  uniformly supercoherent algebra, every element can be written as $\exists_f(\varphi)$, with $\varphi$ $\mathbf{(U\mbox{-}SK_p)}$. In detail, given an object $(B,\psi)$ of $\mathbf{Reg}_{\sP}$, there exists an element  $\varphi\in \sP(A)$ and a morphsim $\freccia{A}{f}{B}$ of $\set$ such that $\beta=\exists_f(\varphi)$. From this, we have that $\varphi \leq \sP_f(\beta)$, i.e. that $\freccia{(A,\varphi)}{f}{(B,\beta)}$ is a morphism in $\Gamma [\sP]$. Therefore, we can define a morphism of $\mathbf{Reg}_{\sP}$, $\freccia{(A,\varphi)}{\phi}{(B,\beta)}$ by $\phi:=\exists_{\angbr{\id_A}{f}}(\varphi)$. By \Cref{regepiasm}, we have that $\phi$ is a regular epi in  $\mathbf{Reg}_{\sP}$ since $\exists_{\pr_B}(\phi)=\exists_f(\varphi)=\beta$. This concludes the proof that $\mathbf{PAsm}_{\sP}$ is a lex (full) sub-category of $\mathbf{Reg}_{\sP}$ and it provides a projective cover of $\mathbf{Reg}_{\sP}$.

    Now we show the other direction. 
The fact that $\mathbf{(U\mbox{-}SK)}$ (or equivalently $\mathbf{(U\mbox{-}SK)_p}$ predicates) elements are closed under finite meets follows by \Cref{lem_PAsm_lex_iff}. Finally, to show that every element of the implicative algebra can be written as $\exists_f(\varphi)$ with $\varphi$ $\mathbf{(U\mbox{-}SK)_p}$, we use the fact that $\mathbf{PAsm}_{\sP}$ provides a protective cover of $\mathbf{Asm}_{\sP}$. In particular, we have that for every element $\beta$ of $\sP(B)$, the object $(B,\beta)$ of $\mathbf{Reg}_{\sP}$ is covered by a regular epi $(\freccia{(A,\varphi)}{\phi}{(B,\beta)}$ where $(A,\varphi)$ is an object of $\mathbf{PAsm}_{\sP}$. Since every object of $\mathbf{PAsm}_{\sP}$ is strongly trackable, we have that $\phi=\exists_{\angbr{\id_A}{f_{\phi}}}(\varphi)$, and since $\phi$ is a regular epi, i.e. $\exists_{\pr_B}(\phi)=\beta$, we can conclude that $\exists_{f_{\phi}}(\varphi)=\beta$. This concludes the proof that $\mathbb{A}$ is uniformly supercompact.
\end{proof}

Given the intrinsic characterization  of the regular completion of a lex category presented in \cite{SFEC}, we have the following corollary: 
\begin{corollary}\label{cor_main_thm}
   Let $\hyperdoctrine{\set}{\sP}$ be an implicative tripos, for a given implicative algebra $\mathbb{A}$. Then  $\mathbb{A}$ is uniformly supercoherent if and only if $\mathbf{PAsm}_{\sP}$ is a lex subcategory of $\mathbf{Reg}_{\sP}$  and $\reglex{\mathbf{PAsm}_{\sP}}\cong \mathbf{Reg}_{\sP}.$
\end{corollary}
\begin{example}
    Relevant examples satisfying the hypotheses of \Cref{cor_main_thm} are implicative algebras associated with a PCA, and implicative algebras associated with a supercoherent locale, see \Cref{ex_superco_realizability_and Heyting_alg}.
\end{example}

As a second corollary of \Cref{thm_main_1}, we obtain a different proof of the characterization of the regular completion of a tripos presented in \cite[Thm. 4.14]{maiettitrotta2021arxiv}:

\begin{corollary}\label{corcor}
Let $\hyperdoctrine{\set}{\sP}$ be an implicative tripos, for a given uniformly supercoherent  implicative algebra $\mathbb{A}$. Then, if $\mathbf{Asm}_{\sP}$ is regular, we have that
\[\reglex{\mathbf{PAsm}_{\sP}}\cong \mathbf{Asm}_{\sP}\cong \mathbf{Reg}_{\sP}.\]
\end{corollary}
Notice that the proof \Cref{thm_main_1} can be reproduced to obtain the following result:
\begin{theorem}
   Let $\hyperdoctrine{\set}{\sP}$ be an implicative tripos, for a given implicative algebra $\mathbb{A}$. If $\mathbf{Asm}_{\sP}$  is regular, then we have that  $\mathbb{A}$ is uniformly functional supercoherent if and only if $\mathbf{PAsm}_{\sP}$ is a lex subcategory of $\mathbf{Asm}_{\sP}$  and $\reglex{\mathbf{PAsm}_{\sP}}\cong \mathbf{Asm}_{\sP}.$
\end{theorem}

\begin{example}
Let $\mathbb{B}$ be a complete atomic boolean algebra which we can think of as the powerset algebra of some set $B$. In $\mathbb{B}$, supercompact elements are atoms which are not closed under infima. This in particular implies that uniformly supercompact predicates are not closed under finite meets.

One can easily see that $\mathbf{PAsm}_{\mathbb{B}}$ is equivalent to the slice category $\set/B$. \footnote{Here and in the following examples we will use subscript $\mathbb{H}$ instead of $\mathsf{P}$ is $\mathsf{P}$ is the implicative tripos arising from the complete Heyting algebra $\mathbb{H}$.}
Although $\set/B$ is clearly a finitely complete category, the inclusion functor into $\Gamma[\sP_{\mathbb{B}}]$ does not preserve finite limits. E.g. the terminal object in $\set/B$ is the identity function from $B$ to $B$ which is set to the assembly $(B,x\mapsto \{x\})$ which is not terminal in $\Gamma[\sP_{\mathbb{B}}]$. 
\end{example}

\begin{example} Let $\mathbb{H}$ be a complete Heyting algebra without supercompact elements. Then $\mathbf{PAsm}_{\mathbb{H}}$ is a trivial category with just one object and the identity arrow.
\end{example}

\begin{example} Consider the Sierpinski space $\mathbf{3}$ which is a supercoherent locale.

It turns out that $\mathbf{PAsm}_{\mathbf{3}}$ is equivalent to the Grothendieck category $\Gamma(\mathbf{Pow})$ of the powerset doctrine $\mathbf{Pow}$ over $\set$, since $1$ and $2$ are supercompact in $\mathbf{3}$ and $1\leq 2$.
Using Corollary \ref{corcor}, we get that $\mathbf{Reg}_{\mathbf{3}}\simeq \reglex{\Gamma[\mathbf{Pow}]}$.

This result can be easily generalized, by considering the locales $\mathbf{n}$ (ordered in the usual way) which are always supercoherent. Since $\mathbf{PAsm}_{\mathbf{n+1}}\simeq \Gamma[\mathsf{P}_{\mathbf{n}}]$, we get that $\mathbf{Reg}_{\mathbf{n+1}}\simeq \reglex{\Gamma[\mathsf{P}_{\mathbf{n}}]}$ for every $n\in \mathbb{N}$.

\end{example}
\subsection{Relation with another notion of assemblies}
In a recent work \cite{CMW}, the authors propose a different notion of implicative assemblies, by generalizing the notion of realizability assemblies in a different direction. Since for realizability implicative algebras $\mathcal{P}(R)\setminus\{\emptyset\}$ is the separator, they define the category of assemblies as the category having as objects pairs $(A,\alpha)$ with $A$ a set and $\alpha:A\rightarrow \Sigma$ and of which an arrow from $(A,\alpha)$ to $(B,\beta)$ is a function $f:A\rightarrow B$ such that $\bigwedge_{x\in A}(\alpha(x)\rightarrow \beta(f(x)))\in \Sigma$. The authors prove that this category is always a quasi-topos. 

In the localic case this category is equivalent to the category of sets, and localic triposes are characterized in \cite{CMW} exactly as those for which such a category is an elementary topos.

One disadvantage of this approach is that the category of assemblies is not in general a full subcategory of the implicative topos. Moreover, it does not contain in general the category of partitioned assemblies as we defined it, since $\mathbf{U}\mbox{-}\mathbf{SK}$ predicates are not in general evaluated in $\Sigma$ (consider e.g.\ the localic case).

The relation with this category of assemblies and our proposal for a category of assemblies is also in general non well-behaved.

\subsection{Categories of implicative modest sets}
Let us end this section by considering the other four full subcategories $\mathsf{Mod}_{\wedge}$, $\mathsf{R}_{\wedge}$, $\mathsf{Mod}_{\times}$ and  $\mathsf{R}_{\times}$ of $\mathbf{Reg}_{\sP}$ of which the objects are assemblies $(A,\alpha)$ of which families corresponding to $\alpha$ can be chosen\footnote{While $\times$-disjointness is a property which is stable under equivalence of families, $\wedge$-disjointness is not in general.} to be $\wedge$-modest families, $\wedge$-core families, $\times$-modest families and  $\times$-core families, respectively. Their inclusions in $\mathbf{Reg}_{\sP}$  factorize through $\mathbf{Asm}_{\sP}$, thus their arrows are always uniquely tracked by functions $f$ between the underlying sets of their domains and codomains. 

Obviously one has the following square of embeddings:

$$\begin{tikzcd}
\mathsf{R}_{\times} & \mathsf{Mod}_{\times} \\
	\mathsf{R}_{\wedge} & \mathsf{Mod}_{\wedge}\\
	\arrow[hook, from=2-1, to=2-2]
 \arrow[hook, from=1-1, to=1-2]
 \arrow[hook, from=1-1, to=2-1]
  \arrow[hook, from=1-2, to=2-2]
\end{tikzcd}$$

Moreover one can easily prove that:
\begin{proposition} $\mathsf{Mod}_{\times}$ and $\mathsf{R}_{\times}$ are preorders.
\end{proposition}
\begin{proof}
Let $(A,\alpha)$ and $(B,\beta)$ be two objects of $\mathsf{Mod}_{\times}$ and let $f,g:A\rightarrow B$ be two functions such that
$$\bigwedge_{x\in A}(\alpha(x)\rightarrow \beta(f(x)))\in \Sigma$$
$$\bigwedge_{x\in A}(\alpha(x)\rightarrow \beta(g(x)))\in \Sigma$$
From this, it follows that
$$\bigwedge_{x\in A}(\alpha(x)\rightarrow \bigexists_{y\in B}\beta(y))\in \Sigma$$
and since
$$\bigwedge_{y,y'\in B}(\beta(y)\times \beta(y')\rightarrow \delta_{B}(y,y'))$$
we can conclude that $f=g$.

Thus $\mathsf{Mod}_{\times}$ is a preorder. Since $\mathsf{R}_{\times}$ is one of its full subcategories, it is a preorder too.
\end{proof}
\begin{example}If $(\Omega,\tau)$ is a topological space and we consider the locale $(\tau, \subseteq)$, then the category $\mathsf{Mod}_{\wedge}=\mathsf{Mod}_{\times}$ is equivalent to the preorder $(\tau,\subseteq)$ itself. Indeed, every open set in $\tau$ is a disjoint union of non-empty connected open sets. Moreover, $\mathsf{R}_{\wedge}=\mathsf{R}_{\times}$ is equivalent to the full sub-poset of $(\tau, \subseteq)$ of which the objects are the open sets which are disjoint unions of supercompact opens.
\end{example}
\begin{example} In the realizability case, $\mathsf{Mod}_{\wedge}$ is a category equivalent to that of modest sets or PERs (see \cite{rosMS}), $\mathsf{R}_{\wedge}$ is equivalent to the category of which the objects are subsets of realizers and of which arrows are functions between them that are restrictions of partial functions which are computable with respect to the combinatory algebra (this category is called $\mathsf{R}$ in \cite{RR}). $\mathsf{Mod}_{\times}$ and $\mathsf{R}_{\times}$ coincide and are equivalent to the partial order $\mathbf{2}$.  
\end{example}

\begin{example} If $\mathbb{B}$ is a complete boolean algebra, then $\mathsf{Mod}_{\wedge}=\mathsf{Mod}_{\times}=\mathsf{R}_{\wedge}=\mathsf{R}_{\times}$ is equivalent to the full sub-preorder of $(\tau, \subseteq)$ of which the objects are those opens which can be written as unions of atoms.
\end{example}
\subsubsection*{Acknowledgements} The authors would like to thank Jonas Frey for fruitful conversations on the topic of the paper and for useful comments on a preliminary version of the present work. The authors are also grateful to Francesco Ciraulo and Milly Maietti for several discussions regarding various aspects of the paper.

\bibliographystyle{plain}
\bibliography{biblio_davide}

\end{document}